\documentclass[10pt, a4paper]{amsart}

\usepackage[T1]{fontenc}
\usepackage{microtype}

\usepackage{amsmath}								%math
\usepackage{amssymb}
\usepackage{amsthm, bm}
\usepackage{amscd}
\usepackage{amsfonts}
\usepackage{stmaryrd}

\usepackage{extarrows}
\usepackage{xypic}

\usepackage[bookmarks=true, colorlinks=true, linkcolor=blue!50!black,
citecolor=orange!50!black, urlcolor=orange!50!black, pdfencoding=unicode]{hyperref}
\usepackage{color}

\usepackage{tikz}									
\usetikzlibrary{matrix}
\usetikzlibrary{patterns}
\usetikzlibrary{matrix}
\usetikzlibrary{positioning}
\usetikzlibrary{decorations.pathmorphing}
\usetikzlibrary{cd}

\usepackage{fullpage}
\usepackage{enumerate}

\newtheorem*{theorem*}{Theorem}

\newtheorem{maintheorem}{Theorem}[section]
\newtheorem{mainquestion}[maintheorem]{Question}

\newtheorem{theorem}{Theorem}[section]

\newtheorem{proposition}[theorem]{Proposition}
\newtheorem{corollary}[theorem]{Corollary} 

\theoremstyle{definition}
\newtheorem{definition}[theorem]{Definition}
\newtheorem{example}[theorem]{Example}

\newtheorem{remark}[theorem]{Remark}

\newtheoremstyle{myitemstyle}						% flexible theorem style
	{}			%Space above
	{}			%Space below
	{}			%Body font
	{}			%indent amount
	{}			%Thm head font
	{.}			%Punkte nach thm head
	{ }			%Abstand nach thm head
	{}			%Thm head spec
\theoremstyle{myitemstyle}
\newtheorem{myitemthm}{}

\newcommand{\hooklongrightarrow}{\lhook\joinrel\longrightarrow}

\newcommand{\R}{\mathbb{R}}

\newcommand{\Z}{\mathbb{Z}}

\newcommand{\Q}{\mathbb{Q}}
\newcommand{\C}{\mathbb{C}}

\newcommand{\PP}{\mathbb{P}}
\renewcommand{\P}{\mathbb{P}}

\newcommand{\G}{\mathbb{G}}

\newcommand{\bbH}{\mathbb{H}}

\newcommand{\Mbar}{\overline{M}}

\newcommand{\calE}{\mathcal{E}}

\newcommand{\calM}{\mathcal{M}}

\newcommand{\calO}{\mathcal{O}}

\newcommand{\bfR}{\mathbf{R}}

\DeclareMathOperator{\Pic}{Pic}
\DeclareMathOperator{\Spec}{Spec}
\DeclareMathOperator{\Hom}{Hom}

\DeclareMathOperator{\pr}{pr}
\DeclareMathOperator{\id}{id}

\DeclareMathOperator{\GL}{GL}

\DeclareMathOperator{\Sym}{Sym}

\DeclareMathOperator{\Jac}{Jac}

\DeclareMathOperator{\End}{End}
\DeclareMathOperator{\Tor}{Tor}
\DeclareMathOperator{\Gr}{Gr}

\DeclareMathOperator{\sd}{sd}

\let\pr\relax
\DeclareMathOperator{\pr}{pr}

\title{$P=W$ phenomena on abelian varieties} 

\date{}

\author{Barbara Bolognese}
\address{Institut f\"ur Mathematik, Goethe--Universit\"at Frankfurt,
60325 Frankfurt am Main, Germany}
\email{bolognese@math.uni-frankfurt.de}

\author{Alex K\"uronya}
\address{Institut f\"ur Mathematik, Goethe--Universit\"at Frankfurt,
60325 Frankfurt am Main, Germany}
\email{kuronya@math.uni-frankfurt.de}

\author{Martin Ulirsch}
\address{Institut f\"ur Mathematik, Goethe--Universit\"at Frankfurt,
60325 Frankfurt am Main, Germany}
\email{ulirsch@math.uni-frankfurt.de}

\begin{document}

\maketitle

\begin{abstract} 
Let $X$ be a complex abelian variety. We prove an analogue of both the (cohomological) $P=W$ conjecture and the geometric $P=W$ conjecture connecting the finer topological structure of the Dolbeault moduli space of topologically trivial semistable Higgs bundles on $X$ and the Betti moduli space of characters of the fundamental group of $X$. The geometric heart of our approach is the spectral data morphism for Dolbeault moduli spaces on abelian varieties that naturally factors the Hitchin morphism and whose target is not an affine space of pluricanonical sections, but a suitable symmetric product.
\end{abstract}

\setcounter{tocdepth}{1}
\tableofcontents

%%%%%%%%%%%%%%%%%%%%%%%%%%%%%%%%%%%%%%%%%%%%%%%%%%%%%%

\section*{Introduction}

The cohomology of complex algebraic varieties carries an intricate mix of additional structures that is determined by it being an algebraic variety. A particularly exciting aspect of this story is the study of moduli spaces associated to the fundamental group of an algebraic variety which originate in non-abelian Hodge theory and, in particular, in the seminal works of Corlette, Donaldson, Hitchin, and Simpson \cite{Corlette, Donaldson, Hitchin,Simpson_nonabelianHodge}. 

The purpose of this article is to study the case of these moduli spaces associated to a fixed abelian variety $X$. 
More specifically, we prove an analogue of the \emph{$($cohomological$)$ $P=W$ conjecture} of de Cataldo--Hausel--Migliorini \cite{dCHM_P=Wn=2} and of the \emph{geometric $P=W$ conjecture} of Katzarkov--Noll--Pandit--Simpson \cite{KNPS} (both originally stated only for curves) for an abelian variety $X$; we note that this is the first class of higher-dimensional varieties where the two $P=W$ conjectures are verified.

Non-abelian Hodge theory provides us with two (or actually three)  moduli spaces of central importance associated to a smooth projective algebraic variety $X$: the \emph{Betti moduli space} $M_\textrm{Betti}^r(X)$ of characters of the fundamental group and the \emph{Dolbeault moduli space} $M_{\textrm{Dol}}^r(X)$ parametrizing Higgs bundles on $X$ (as well as the de Rham moduli space which will play no role in the remainder of this article).  
Both of these moduli spaces are quasi-projective algebraic varieties which are very much not isomorphic as complex algebraic varieties (one being affine, the other one not). Nevertheless Simpson's work \cite{SimI, SimII} provides us with a real analytic isomorphism between $M_\textrm{Betti}^r(X)$ and $M_{\textrm{Dol}}^r(X)$. 

The Dolbeault moduli space has a distinctive feature: it comes with the \emph{Hitchin map} \cite{Hitchin},  a proper map to a base space that is constructed from global sections.  One may think of this as an `abelianization' process or  a form of Strominger--Yao--Zaslow  mirror symmetry. It is natural for the geometry of the Hitchin map to be seen in the cohomology ring of $M^r_{\textrm{Dol}}(X)$, but much less so in the cohomology of the Betti moduli space $M^r_{\textrm{Betti}}(X)$. Such phenomena were initially observed over curves by Hausel--Rodriguez-Villegas \cite{HauselRodriguezVillegas} under the heading `curious Poincar\'e duality', which, in an effort to explain these phenomena, then led to the \emph{$P=W$ conjecture} of de Cataldo--Hausel--Migliorini in \cite{dCHM_P=Wn=2}. The latter predicts that the real analytic isomorphism between Betti and Dolbeault moduli space swaps the weight filtration on the cohomology of $M_{\textrm{Betti}}^r(X)$ (coming from mixed Hodge theory) with the perverse filtration on the cohomology of $M_{\textrm{Dol}}^r(X)$ (coming from the Hitchin fibration). 

The \emph{geometric $P=W$ conjecture} of Katzarkov--Noll--Pandit--Simpson \cite{KNPS} is an effort to geometrically explain the surprising equality in the cohomological $P=W$ conjecture by comparing the geometry of neighborhoods at infinity of both moduli spaces that keep track of their quite different  compactifications. It is now known by \cite[Theorem A]{MMS_geometricP=W} that, in the curve case and under some technical assumptions, the geometric $P=W$ conjecture implies the top weight part of the cohomological $P=W$ conjecture.

\subsection*{The cohomological $P=W$ conjecture for compact Riemann surfaces} More concretely, let $X$ be a compact Riemann surface of genus $g\geq 1$ and $p\in X$ a fixed point. Fix a degree $d\geq 0$ and primitive $d$-th root of unity $\mu_d$. The non-abelian Hodge correspondence \cite{Simpson_nonabelianHodge,SimI, SimII} provides us with a real-analytic isomorphism between the Dolbeault moduli space $M_{\textrm{Dol}}^{r,d}(X)$ of semistable Higgs bundles of rank $r\geq 1$ and degree $d$ on $X$ and the Betti moduli space of local systems of rank $r$ on $X$ with fixed monodromy equal to $\mu_d I_r$ around $p$. In other words the latter is the GIT quotient 
\begin{equation*}
M_{\textrm{Betti}}^{r,d}(X)=\big\{A_k, B_k\in\GL_r(\mathbb{C}) \textrm{ for } k=1,\ldots, g\ \big\vert \ [A_1,B_1]\cdots [A_g, B_g]= \mu_d I_r\big\}\sslash_{\GL_r{\C}} \ .
\end{equation*}
This provides us with an identification between the cohomology rings
\begin{equation*}
H^\ast\big(M_{\textrm{Dol}}^{r,d}(X),\Q\big)=H^\ast\big(M_{\textrm{Betti}}^{r,d}(X),\Q\big) 
\end{equation*}
of these a priori (and factually) quite different algebraic varieties.

The $P=W$ conjecture of de Cataldo--Hausel--Migliorini \cite{dCHM_P=Wn=2} predicts an unexpected equality 
\begin{equation*}
P_{k}H^\ast\big(M_{\textrm{Dol}}^{r,d}(X), \Q\big)=W_{2k} H^\ast \big(M_{\textrm{Betti}}^{r,d}(X),\Q\big)=W_{2k+1} H^\ast \big(M_{\textrm{Betti}}^{r,d}(X),\Q\big)
\end{equation*}
between the perverse filtration on $H^\ast\big(M_{\textrm{Dol}}^{r,d}(X),\Q\big)$ associated to the Hitchin fibration, and the weight filtration on $H^\ast\big(M_{\textrm{Betti}}^{r,d}(X),\Q\big)$ in the sense of Deligne \cite{Deligne_HodgeII, Deligne_HodgeIII} (at least when $d$ and $r$ are coprime). The $P=W$ conjecture has been proved in the case $r=2$ and $d=1$ and arbitrary genus $g\geq 2$ in \cite{dCHM_P=Wn=2} and for $g=2$ and arbitrary prime rank $r\geq 1$ coprime to $d$ in \cite{dCMS_g=2pprime}. Two recent preprints \cite{MaulikShen_smoothP=W, HMMS_P=W} contain proofs of the $P=W$ conjecture for all $g\geq 2$ and $r\geq 1$ coprime to $d$ (also see \cite{MSY} for a third proof and \cite{Felisetti_P=Wsurvey} for an overview of recent developments). The case when $r$ and $d$ are not coprime features non-smooth moduli spaces and has its own set of $P=W$ conjectures on the level of intersection cohomology (see e.g.\  \cite{FelisettiMauri} for details).

\subsection*{A cohomological $P=W$ phenomenon for abelian varieties} The main result  of this article is  an analogous $P=W$ phenomenon in the case where we replace the compact Riemann surface by a complex abelian variety $X=\mathbb{C}^g/L$ of dimension $g\geq 1$. We shall see in Sections \ref{section_Dolbeaultmoduli} and \ref{section_nonabelianHodge} below that both the Dolbeault moduli space $M_{\textrm{Dol}}^r(X)$ of topologically trivial rank $r$ semistable Higgs bundles and the Betti moduli space 
\begin{equation*}
M_{\textrm{Betti}}^r(X)=\Hom\big(L,\GL_r(\mathbb{C})\big)\sslash_{\GL_r{\C}}
\end{equation*}
have the structure of a suitable $r$-th symmetric power. This allows us to explicitly write down the non-abelian Hodge correspondence in Section  \ref{section_explicitNAH}, which will give rise to a real-analytic isomorphism  $\eta^r_X\colon M_{\textrm{Dol}}^r(X)\xrightarrow{\sim} M_{\textrm{Betti}}^r(X)$ and, in turn, an identification
\begin{equation}\label{eq_cohomology}
H^\ast\big(M_{\textrm{Dol}}^r(X),\Q\big)=H^\ast\big(M_{\textrm{Betti}}^r(X),\Q\big) \ .
\end{equation}

As we will see in Section \ref{section_Hitchin} below, the Hitchin morphism naturally factors through the so-called \emph{spectral data morphism} 
\begin{equation*}\label{eq_Hitchinfibration}
\sd_X\colon M^r_{\textrm{Dol}}(X)\longrightarrow \Sym^r\C^g
\end{equation*}
(in the terminology of \cite{CN_Hitchin}), whose base is the $r$-th symmetric power of the affine space $\mathbb{C}^g$ (and not an affine space). The spectral data morphism is proper. Its fibers are natural symmetric powers of the dual abelian variety $\widehat{X}$ to $X$. The Dolbeault moduli space naturally carries the structure of a symplectic orbifold making $\sd_X$ into an orbifold Lagrangian fibration (see Proposition \ref{prop_Hitchin} below for details). 

As explained in Section \ref{section_perversefiltration} below, the spectral data morphism allows us to endow $H^\ast\big(M_{\textrm{Dol}}^r(X,\Q)\big)$ with a perverse filtration
\begin{equation*}
P_kH^\ast\big(M_{\textrm{Dol}}^r(X),\Q\big) \quad \textrm{ for } \quad k\in \Z
\end{equation*}
with $P_kH^\ast\big(M_{\textrm{Dol}}^r(X),\Q\big)=0$ for $k< 0$ and $P_kH^\ast\big(M_{\textrm{Dol}}^r(X),\Q\big)=H^\ast\big(M_{\textrm{Dol}}^r(X),\Q\big)$ for $k\geq g$.

On the other side, the Betti moduli space $M_{\textrm{Betti}}^r(X)$ is affine, and has at most quotient singularities. Therefore, using Deligne's classical results \cite{Deligne_HodgeII,Deligne_HodgeIII} the cohomology ring $H^\ast\big(M_{\textrm{Betti}}^r(X),\Q\big)$ carries a natural weight filtration 
\begin{equation*}W_k H^\ast\big(M_{\textrm{Betti}}^r(X),\Q\big) \quad\textrm{ for } \quad k\in\Z 
\end{equation*}
with $W_k=0$ for $k<0$ and $W_k=H^\ast\big(M_{\textrm{Betti}}^r(X),\Q\big)$ for $k\geq 2g$.  

The following Theorem \ref{mainthm_P=W} establishes a (cohomological) $P=W$ phenomenon in our situation.

\begin{maintheorem}[P=W for abelian varieties]\label{mainthm_P=W}
Let $X$ be a complex abelian variety of dimension $g\geq 1$. Then, under the identification \eqref{eq_cohomology} 
induced by the non-abelian Hodge correspondence, we have
\begin{equation*}
P_{k} H^\ast\big(M_{\textrm{Dol}}^r(X),\Q\big)=W_{2k}H^\ast\big(M_{\textrm{Betti}}^r(X),\Q\big) =W_{2k+1}H^\ast\big(M_{\textrm{Betti}}^r(X),\Q\big)
\end{equation*}
for all $k\in\Z$.
\end{maintheorem}

A key idea of the proof is to identify both moduli spaces with suitable symmetric products and to use results on the compatibility of symmetric products with both the weight filtration and the perverse filtration. %this allows us to reduce to the rank one case. %The latter compatibility is very much a non-trivial statement that turns out to work thanks to results by of Zhang \cite{ZHA}. 
Similar techniques also appear in \cite{FelisettiMauri} (in particular in the proof of \cite[Theorem 7.6]{FelisettiMauri}), where the authors study $P=W$ phenomena for character varieties that admit a symplectic reduction and, in particular, the case of $A$ being an elliptic curve. The main difference from this earlier approach is that we make use of the spectral data morphism instead of the Hitchin morphism, which allows us to go beyond the case of a one-dimensional $X$. 

Theorem \ref{mainthm_P=W} together with the results in \cite{FlorentinoSilva} and the relative hard Lefschetz theorem for the $P$-filtration implies the following:

\begin{maintheorem}[Curious Poincar\'e and curious hard Lefschetz]\label{mainthm_curious}
Let $X$ be a complex abelian variety of dimension $g\geq 1$. 
\begin{enumerate}[(i)]
\item Then $M_{\textrm{Betti}}^r(X)$ is of Hodge--Tate type and we have
\begin{equation*}
H\big(M_{\textrm{Betti}}^r(X);\frac{1}{qt^2},t\big)=(qt)^{-2gr}\cdot H\big(M_{\textrm{Betti}}^r(X);q,t)
\end{equation*}
for the mixed Hodge polynomial $H(.;q,t)=H(.;\sqrt{q},\sqrt{q},t)$ of $M_{\textrm{Betti}}^r(X)$. 
\item Let $L$ be a hyperplane class in $H^2\big(M_{\textrm{Dol}}^r(X),\Q\big)\cong H^2\big(M_{\textrm{Betti}}^r(X),\Q\big)$. Then the $k$-fold multiplication with $L$ induces an isomorphism
\begin{equation*}
L^k\colon \Gr_{2gr-2k}^WH^\ast\big(M_{\textrm{Betti}}^r(X),\Q\big)\xlongrightarrow{\sim}\Gr_{2gr+2k}^WH^{\ast+2k}\big(M_{\textrm{Betti}}^r(X),\Q\big) \ .
\end{equation*}
\end{enumerate}
\end{maintheorem}

Theorem \ref{mainthm_curious} is an analogue of the curious Poincar\'e duality and the curious hard Lefschetz theorem that have originally been discovered in \cite{HauselRodriguezVillegas} in the rank $2$ case. 

\subsection*{A geometric $P=W$ phenomenon for  abelian varieties} 
In \cite{KNPS} the authors propose a geometric incarnation of the $P=W$ phenomenon on a compact Riemann surface that suggests a correspondence between the geometry at infinity of both the Dolbeault and the Betti moduli (see \cite[Section 5]{Migliorini_survey} and \cite{MMS_geometricP=W} for more details). 

Let $X$ be a complex abelian variety of dimension $g$. Based on our explicit description of the non-abelian Hodge correspondence in Section \ref{section_explicitNAH} below, we establish a geometric $P=W$ phenomenon on $X$.

As we shall see in Section \ref{section_Hitchin} below, and inspired by \cite{Hausel_compactification, Simpson_Hodgefiltration, deCataldo_compactification}, the Dolbeault moduli space $M_{\textrm{Dol}}^r(X)$ admits a compactification $\Mbar_{\textrm{Dol}}^r(X)$ together with a regular extension of the spectral data morphism. We write $N_{\textrm{Dol}}^r(X)$ for the intersection of a suitable open neighborhood of the boundary of $\Mbar_{\textrm{Dol}}^r(X)$ and we will refer to $N_{\textrm{Dol}}^r(X)$ as a \emph{neighborhood of infinity}. We also point out that the spectral data base $\Sym^r \big(\C^g\big)$ of $M_{\textrm{Dol}}^r(X)$ naturally contains the sphere quotient $S^{2gr-1}/S_r$.

In general, the Betti moduli space $M_{\textrm{Betti}}^r(X)$ does not admit a standard compactification. Thanks to the results of \cite{Stepanov, Thuillier}, however, the homotopy type of the dual complex of a normal crossing compactification of $M_{\textrm{Betti}}^r(X)$ does not depend on the choice of the compactification (see \cite[Theorem 1.2]{Harper} for an orbifold version). Since the Betti moduli space $M_{\textrm{Betti}}^r(X)$ is an open Calabi-Yau orbifold, we can do even better: Once we know $M_{\textrm{Betti}}^r(X)$ admits a divisorial log terminal (dlt) compactification, all other dlt compactifications that are crepant birational have the same dual complex up to piecewise linear homoemorphism (by \cite[Proposition 11]{dFKX_dualcomplexes}). So, by a slight abuse of notation, it makes sense to denote the dual complex simply by $\Delta\big(M_{\textrm{Betti}}^r(X)\big)$. It has been observed in \cite[Section 2]{MMS_geometricP=W} that we may identify $\Delta\big(M_{\textrm{Betti}}^r(X)\big)$ with the link of the essential skeleton of $M_{\textrm{Betti}}^r(X)$ in the sense of \cite{KontsevichSoibelman, MustataNicaise, NicaiseXuYu}. As above, we write $N_{\textrm{Betti}}^r(X)$ for a suitable neighborhood at infinity.

\begin{maintheorem}[Geometric $P=W$ for abelian varieties]\label{mainthm_geometricP=W}
Let $X$ be a complex abelian variety of dimension $g\geq 1$. The Betti moduli space $M_{\textrm{Betti}}^r(X)$ admits a dlt compactification and there is natural piecewise linear homoemorphism between the dual complex $\Delta\big(M_{\textrm{Betti}}^r(X)\big)$ and $S^{2gr-1}/S_r$ that makes the natural diagram
\begin{equation}\label{eq_geometricP=W}\begin{tikzcd}
N_{\textrm{Dol}}^r(X) \arrow[rr,"\sim","\eta_X^r"']\arrow[d,"\widetilde{\sd}_X"']&& N_{\textrm{Betti}}^r(X) \arrow[d,"\rho"]\\
S^{2gr-1}/S_{r}\arrow[rr,"\sim"]&& \Delta\big(M_{\textrm{Betti}}^r(X)\big)
\end{tikzcd}\end{equation}
commute, where:
\begin{itemize}
\item $N_{\textrm{Dol}}^r(X)$ and $N_{\textrm{Betti}}^r(X)$ are suitable neighborhoods of infinity on which $\eta_X^r$ is a homeomorphism;
\item the map $\widetilde{\sd}_X$ is the composition of the spectral data morphism $\sd_X$ with the  retraction from $\Sym^r(\C^g)-\big\{[0]\big\}$ to $S^{2gr-1}/S_r$ along rays through the origin; and 
\item the map $\rho$ is a suitable retraction onto $\Delta\big(M_{\textrm{Betti}}^r(X)\big)$.
\end{itemize}
\end{maintheorem}

In the case of a compact Riemann surface the arrows in the  diagram analogous to \eqref{eq_geometricP=W} are only homotopy equivalences and the retraction $\rho$ is only well-defined up to homotopy. For abelian varieties Diagram \eqref{eq_geometricP=W} commutes without the need for homotopy; heuristically this comes from the fact that the same is true in the case of elliptic curves (see \cite[Section 5]{MMS_geometricP=W}) and the case of abelian varieties generalizes this situation.

\subsection*{Towards $P=W$ phenomena in higher dimension} So far, $P=W$ phenomena seem to have only appeared for various versions of Dolbeault and Betti moduli spaces on compact Riemann surfaces, and for generalizations of these moduli spaces, namely for certain types of cluster varieties \cite{Zhang_cluster} and hyperk\"ahler varieties \cite{HLSY_hyperKaehlerP=W} (also see \cite[Section 4.4]{dCHM_exchange} for some speculations in that direction). 

Theorem \ref{mainthm_P=W} and \ref{mainthm_geometricP=W} give a complete description of the nature of the cohomological and the geometric $P=W$ phenomenon when $X$ is an abelian variety. They form the first step towards  the question whether $P=W$ phenomena occur for Betti/Dolbeault moduli spaces of other algebraic varieties.

The following crucial points need to be considered when searching for such a generalization:

\begin{enumerate}

\item The factorization of the Hitchin fibration into the spectral data morphism is a special case of a factorization of the Hitchin morphism for higher-dimensional varieties coming from \cite{CN_Hitchin}. An in-depth  geometric understanding of spectral data morphism will be of central importance.

\item The appearance of a sphere in the Hitchin base, as phrased in \cite{KNPS}, appears to be particular to the curve case. In Remark \ref{remark_notasphere} below we explain why $S^{2gr-1}/S_r$ is not a sphere for $g,r\geq 2$, but only a $\Q$-homology sphere. 
This aligns with a general conjecture of Koll\'ar and Xu \cite[Question 4]{KollarXu} on the homotopy type of dual complexes of open Calabi-Yau varieties.

\item Theorem \ref{mainthm_geometricP=W} could be phrased in terms of essential skeletons and, in general, the non-Archimedean approach to the SYZ fibration in the sense of Kontsevich--Soibelman \cite{KontsevichSoibelman} (also see \cite{MustataNicaise, NicaiseXuYu, MMS_geometricP=W}). This illustrates the importance of interactions with techniques coming from SYZ mirror symmetry.
\end{enumerate}

What the case of compact Riemann surfaces (of genus $\geq 1$) and abelian varieties have in common is that they are both $K(\pi,1)$-varieties, i.e.\ complex algebraic varieties whose universal cover is contractible. 

\begin{mainquestion}
Do $P=W$ phenomena occur for Betti/Dolbeault moduli spaces over $K(\pi,1)$-varieties? 
\end{mainquestion}

The cohomology of a $K(\pi,1)$-variety is naturally isomorphic to the group cohomology of its fundamental group. So, heuristically, one may think of $K(\pi,1)$-varieties as those varieties whose topology is determined by their fundamental group. We believe that $K(\pi,1)$-varieties form a natural class of higher-dimensional algebraic varieties, on which the search for $P=W$ phenomena could be successful.

\subsection*{Acknowledgements} We thank Emilio Franco, Andreas Gross, Johannes Horn, Inder Kaur, Martina Lanini, Mirko Mauri, Enrica Mazzon, and Annette Werner for insightful discussions and remarks at different stages of the project. Particular thanks are due to Camilla Felisetti for allowing us to discuss the details of our approach with her and to Matthias Kreck, who communicated the argument in Remark \ref{remark_notasphere} to us. 

This project has received funding from the Deutsche Forschungsgemeinschaft  (DFG, German Research Foundation) TRR 326 \emph{Geometry and Arithmetic of Uniformized Structures}, project number 444845124, from the Deutsche Forschungsgemeinschaft (DFG, German Research Foundation) Sachbeihilfe \emph{From Riemann surfaces to tropical curves (and back again)}, project number 456557832, and from the LOEWE grant \emph{Uniformized Structures in Algebra and Geometry}.

%\setcounter{tocdepth}{1}
%\tableofcontents

%%%%%%%%%%%%%%%%%%%%%%%%%%%%%%%%%%%%%%%%%%%%%%%%%%%%%%

\section{The Dolbeault moduli space on a complex abelian variety}\label{section_Dolbeaultmoduli}

In this section, we will give an explicit description of the moduli space of semistable Higgs bundles on an abelian variety. 

\subsection{Preliminaries on Higgs bundles and (semi-)stability}

Let $X$ be an $n$-dimensional smooth projective complex algebraic variety. We write $\Omega^1_X$ for the cotangent bundle of $X$ and always implicitly fix a hyperplane class $H$. Our results will not depend on the choice of $H$.

\begin{definition}
A \textit{Higgs sheaf} on $X$ is  a pair $(E,\phi)$, where $E$ is a coherent sheaf on $X$, and $\phi\in H^0(X, \calE nd(E)\otimes \Omega ^1_X)$ such that $\phi \wedge \phi =0$. Whenever $E$ is locally free, we will call such a pair a {\it Higgs bundle}.
\end{definition}

The natural notion of stability for Higgs sheaves is very similar to the notion of slope stability for torsion-free coherent sheaves. Recall that, if $E$ is a torsion-free sheaf, its {\it slope} (with respect to $H$) is defined as 
\[ 
\mu(E) =   \mu_H(E) = \frac{\text{deg}(E)}{\text{rk}(E)} = \frac{c_1(E)\cdot H^{n-1}}{\text{rk}(E)}  \ .
\]

\begin{definition}
A subsheaf $F \subset E$ is said to be \emph{$\phi$-invariant} if $\phi(F) \subset F\otimes \Omega^1_X$. A Higgs sheaf is said to be \emph{semistable}, if $\mu(F) \leq \mu(E)$. If, in addition, we have $\mu(F) < \mu(E)$ for any nonzero $\phi$-invariant subsheaf $F\subsetneq E$, we say that $E$ is  \emph{stable}. Finally, $(E, \phi)$ is called \emph{polystable} if it is the direct sum of stable $\phi$-invariant subsheaves, all of the same slope.
\end{definition}

In what follows we say that a vector bundle $E$ on $X$ is \emph{topologically trivial}, if all of its Chern classes in $H^\ast(X,\Q)$ %(or, equivalently, in $\CH^\ast(X)$) 
are trivial.

\begin{remark}
    If a semistable torsion free Higgs sheaf $(E,\phi)$ is topologically trivial, then $E$ is a bundle by \cite[Proposition 6.6]{SimII}; it is in fact an extension of stable Higgs bundles whose Chern classes all vanish. Any sub-Higgs sheaf of slope zero is a strict subbundle with vanishing Chern classes. By \cite[Proposition 6.6]{SimII} we also know that for topologically trivial Higgs bundles  Gieseker (semi-)stability is equivalent to slope (semi-)stability. We work with slope (semi-)stability for the sake of definiteness; but all of our results remain valid for Gieseker (semi-)stability as well.
\end{remark}

A crucial tool when working with Higgs sheaves are the following two filtrations that filter our sheaves by semistable or respectively stable quotients:

\begin{itemize}
\item Let $E$ be a torsion-free Higgs sheaf. The {\it Harder--Narasimhan filtration of $E$}  is the unique filtration
\begin{equation*}
    0 = F_0 \subsetneq F_1 \subsetneq \ldots\subsetneq F_{s-1} \subsetneq F_s = E
\end{equation*}
with $\phi$-invariant subsheaves $F_i$ and semistable subquotients $E_i:=F_i/F_{i=1}$ for $1\leq i\leq s$ such that
$\mu(E_i)>\mu(F_j)$ whenever $i<j$.

\item Let $F$ be a semistable torsion-free Higgs sheaf. The  {\it Jordan--Hölder of $E$} filtration is the filtration
\begin{equation*}
    0 = G_0 \subsetneq G_1 \subsetneq \ldots\subsetneq G_{s-1} \subsetneq G_s = F
\end{equation*}
with $\phi$-invariant subsheaves $G_i$, and semistable subquotients $H_i:=G_i/G_{i=1}$ for $1\leq i\leq s$ such that
$\mu(H_i)=\mu(H_j)$ for each $i$ and $j$.
\end{itemize}

Two semistable Higgs sheaves $E,E'$ are called \emph{S-equivalent}, written as $E\sim_SE'$, if their respective Jordan--Hölder factors are isomorphic. Every semistable Higgs bundle is S-equivalent to a polystable Higgs bundle: indeed, if $(F, \phi)$ is a semistable Higgs bundle with Jordan--Hölder factors $H_1,...,H_n$, then it is S-equivalent to the polystable Higgs bundle $(H_1, \phi|_{H_1})\oplus ... \oplus (H_n, \phi|_{H_n})$.

 \subsection{(Semi-)Stability of Higgs bundles on abelian varieties} In this section we give a complete characterization of topologically trivial semistable Higgs bundles on complex abelian varieties.

\begin{proposition}\label{prop_semistableHiggs}
Let $X$ be a complex abelian variety of dimension $g$. Then a Higgs bundle $(E,\Phi)$ on $X$ is semistable if and only if the underlying vector bundle $E$ is semistable. 
\end{proposition}

Proposition \ref{prop_semistableHiggs} expands on Proposition 4.1 of \cite{FGPN}, which deals with the case of an elliptic curve. Before we start the proof, we note that the following: If $F\subseteq E$ is a subbundle, then by definition $F$ is $\phi$-invariant if $\phi(F)\subseteq F\otimes \Omega_X^1$. Since on an abelian variety we have $\Omega_X^1\cong \calO_X^{\oplus g}$ this condition becomes
\[
\phi(F)\subseteq F^{\oplus g}\subseteq E^{\oplus g}\ . 
\]

\begin{proof}[Proof of Proposition \ref{prop_semistableHiggs}]
If $E$ is semistable then every subbundle has slope at most that of $F$, hence the Higgs bundle $(E,\Phi)$ is necessarily semistable. 

Conversely, assume that $E$ is \emph{not} semistable, let us write 
\[
0 = F_0 \subsetneq F_1 \subsetneq \ldots\subsetneq F_{s-1} \subsetneq F_s = E  
\]
for its Harder--Narasimhan filtration. Recall that the subquotients $E_i:=F_i/F_{i=1}$ for $1\leq i\leq s$ are all semistable and 
$\mu_\alpha(E_i)>\mu_\alpha(F_j)$ whenever $i<j$. In addition
\[
H^0(X,\Hom_{\calO_X}(E_i,E_j)) \,=\, 0\ \ \text{for all $i<j$.}
\]
Consider the maximal destabilizing subbundle $F_1$ of $E$. If $F_1\subseteq\ker \phi$ then $F_1$ is $\Phi$-invariant with slope strictly larger than that of $E$, therefore the Higgs bundle$(E,\phi)$ is not semistable. We can therefore assume without loss of generality that $\Phi(F_1)\neq 0$. Let $0\leq l\leq s$ be the smallest index for which
$\phi(F_1)\, \subseteq\, F_l$ but also $\phi(F_1)\,\not\subseteq\, F_{l-1}$. 
By the choice of $l$, the Higgs field $\phi$ induces a non-zero morphism 
\[
\phi_1 \colon E_1 = F_1 \longrightarrow \left(F_l^{\oplus g}\right) / \left(F_{l-1}^{\oplus g}\right) %
\cong \left( F_l/F_{l-1}\right)^{\oplus g} = E_l^{\oplus g} 
\]
However, the right-hand side is semistable with slope/Hilbert polynomial  
\[
\mu \left(E_l^{\oplus g}\right) = \mu(E_l) < \mu(E_1)\ ,
\]
which is impossible unless $l=1$. This implies $\phi(F_1)\subseteq F_1^{\oplus g}$, which means that $F_1$ is $\phi$-invariant.
\end{proof}

\begin{proposition}\label{prop_stableHiggs}
A topologically trivial semistable Higgs bundle on a complex abelian variety $X$ is stable if and only if it is of rank one. 
\end{proposition}

Our proof is, in parts, a generalization of \cite[Proposition 4.3]{FGPN}, which proves an analogous result in the case when $X$ is an elliptic curve.

\begin{proof}[Proof of Proposition \ref{prop_stableHiggs}]
Let $(E,\phi)$ be an topologically trivial stable Higgs bundle on $X$. By Proposition \ref{prop_semistableHiggs} the underlying vector bundle $E$ is a semistable vector bundle. By \cite{MehtaNori, Langer_boundedness} topologically trivial semistable vector bundles on $X$ are exactly the homogeneous bundles on $X$; so we have 
\[
t_x^\ast E\cong E
\]
for all translation maps $t_x$ by $x\in X$. The  classification of homogeneous vector bundles  \cite{Matsushima,Miyanishi, Morimoto, Mukai_semihomogeneous} now tells us that
\begin{equation}\label{eq_homogeneousdecomposition}
    E\cong \bigoplus_i L_i\otimes U_i
\end{equation}
where the $L_i\in\Pic_0(X)$ are line bundles with $c_1(L_i)=0$, and the $U_i$'s are unipotent vector bundles. 

We may rewrite \eqref{eq_homogeneousdecomposition} as 
\begin{equation*}
    E\cong \bigoplus_j L_j'\otimes U_j'
\end{equation*}
where $L'_j$ and $L_k'$ are not isomorphic  whenever $j\neq k$ and $U_j'$ is a direct sum of unipotent bundles. Observe that 
\begin{equation}\label{eq_usefuldecomposition}
    H^0\big(X,\Hom(L_j,L_k)\big)=H^0(X,L_j^\ast\otimes L_k)=0
\end{equation}
whenever $j\neq k$, since all Chern classes of $L_j^\ast\otimes L_k$ are trivial, but $L_j^\ast \otimes L_k$ is not the trivial bundle. 

Hence the homomorphism $\varphi\colon E\rightarrow E$ decomposes into a direct sum of endomorphisms $\phi_i\colon L_j'\otimes U_j'\rightarrow L_j'\otimes U_j'$. This contradicts the stability of the Higgs bundle $(E,\phi)$ unless $s=1$. Therefore we have 
\begin{equation*}
    E\cong L\otimes U
\end{equation*}
where $L\in\Pic_0(X)$ and $U$ is a direct sum of unipotent bundles. 

There are natural isomorphisms
\begin{equation*}
    \End(E)\cong U^\ast\otimes L^\ast\otimes L\otimes U\cong U^\ast\otimes U\cong \End(U) \ .
\end{equation*}
Since $U$ is a direct sum of unipotents, there is a unique trivial subbundle $\calO_X^k\hookrightarrow U$ of maximal rank $k$. The homomorphism $H^0(X,\calO_X^k)\hookrightarrow H^0(X,U)$ is injective. Since $\phi$ sends non-zero sections of $\calO_X^k$ either to zero or to non-zero sections of $U$ and so we find that $\phi(\calO_X^k)\subseteq \calO_X^k$. The existence of a such a subbundle contradicts the stability of $(E,\phi)$, unless $\calO_X^k\cong U$. Therefore we have 
\begin{equation*}
    \End(E)\cong \End(U,U)\cong \End(\calO_X^k,\calO_X^k)=\mathbb{C}^{k\times k} \ .
\end{equation*}
Choose a non-trivial eigenspace $V$ of $\phi$. Then $L\otimes V$ is $\phi$-invariant and the existence of $V$ contradicts the stability of $(E,\phi)$ unless $k=1$. But this implies that $E\cong L$ and so the Higgs bundle $(E,\phi)$ has rank one. 

The converse direction is immediate. 
\end{proof}

\subsection{The moduli space of semistable Higgs bundles}

In \cite{SimII} Simpson constructs a Dolbeault moduli space $M_{Dol}^r(X)$ of topologically trivial semistable Higgs bundles of rank $r$ on $X$. It corepresents the moduli functor
\begin{equation*}
\calM_{\textrm{Dol}}^r(X)\colon \mathbf{Sch}_{\C}^{op}\longrightarrow \mathbf{Sets}
\end{equation*}
associating to a $\C$-scheme $S$ the set of isomorphism classes of topologically trivial semistable Higgs sheaves on $X\times S$ of rank $r$. The closed points of $M_{\textrm{Dol}}^r(X)$ correspond to $S$-equivalence classes of topologically trivial semistable Higgs bundles of rank $r$ on $X$ by \cite[Proposition 6.6]{SimII}.

For the rest of this subsection let $X$ be an abelian variety of dimension $g\geq 1$, and write $\widehat{X}$ for the dual abelian variety.

\begin{proposition}\label{prop_modulispace}
With notation as above, the Dolbeault moduli space $M_{\textrm{Dol}}^r(X)$ is naturally isomorphic to the symmetric product $\Sym^r(\widehat{X}\times\C^g)$.
\end{proposition}

Note that  a principal polarization on $X$ gives rise to a natural isomorphism $X\simeq \widehat{X}$, and thus an isomorphism
\begin{equation*}
M_{\textrm{Dol}}^r(X)\cong \Sym^r(\widehat{X}\times\C^g)\cong \Sym^r(X\times\C^g)\ .
\end{equation*}
Proposition \ref{prop_modulispace} is a special case of \cite[Theorem 3.14]{FT}. We sketch the main idea of this proof in the language of this article to be able to use it in Section \ref{section_nonabelianHodge}.

\begin{proof} 
Consider a topologically trivial semistable Higgs bundle $(E,\phi)$ of rank $r$ on $X$. Since $(E, \phi)$ is topologically trivial, it is  S-equivalent to the direct sum of topologically trivial stable Higgs bundles, which are of rank one by Proposition \ref{prop_stableHiggs} above. So we have $(E,\phi)\sim _S (L_1, \phi|_{L_1})\oplus  ...\oplus (L_r, \phi|_{L_r})$. 

For each $i \in \{0,...,r\}$, one has that $\phi|_{L_i} \in \mathrm{Hom}(L_i, L_i \otimes \Omega_X^i ) \cong \mathrm{Hom}(L_i, L_i ^{\oplus }) \cong \mathbb{C}^n$ hence it can be represented by a string of complex numbers ${\bm \lambda}_i= (\lambda _{i,1}, ..., \lambda_ {i,g})$. So we can define a map
\begin{equation*}
t_\C\colon M _{\mathrm{Dol}}^r(X) \longrightarrow \mathrm{Sym}^r(\widehat{X}\times \mathbb{C}^g)
\end{equation*}
by setting
\begin{equation*}
    t_\C (E,\phi) = \big[ (L_1, \bm{\lambda}_1), ..., (L_r, \bm{\lambda}_r)\big].
\end{equation*}
We define an inverse 
\begin{equation*} t_\C^{-1} : \text{Sym}^r(\widehat{X}\times \mathbb{C}^g) \longrightarrow M_{\textrm{Dol}}^r(X) 
\end{equation*}
by the map
\begin{equation*}
\big[ (L_1, \bm{\lambda}_1), ..., (L_r, \bm{\lambda}_r)\big] \longmapsto  ( L_1 \oplus \cdots \oplus L_r, \bm{\lambda}_1 \text{Id}_{L_1} \oplus \bm{\lambda}_r \text{Id}_{L_r}) , 
\end{equation*}
where $( L_1 \oplus \cdots \oplus L_r, \bm{\lambda}_1 \text{Id}_{L_1} \oplus \bm{\lambda}_r \text{Id}_{L_r})$ is the polystable element in the S-equivalence class of $(E, \phi)$.

This way, $t_\C$ defines a bijection between the Dolbeault moduli space $M_{\textrm{Dol}}^r(X)$ and $\Sym^r\big(\widehat{X}\times\C^g\big)$. In order to see that this bijection is actually an isomorphism of algebraic varieties, we need to observe that there is a natural transformation
\begin{equation*}
t: \calM _{\mathrm{Dol}}^r(X) \longrightarrow \Hom\big(-,\Sym^r(\widehat{X}\times \C^g)\big) \end{equation*}
that is given by $t_\C$ on $\Spec \C$-valued points. As explained in the proof of \cite[Theorem 3.14]{FT}, the Fourier--Mukai transform for $\Lambda$-modules in the sense of \cite{PolishchukRothstein} provides us with a way to functorially associate to a topologically trivial semistable Higgs sheaf on $X\times S$ of rank $r$ (for a $\C$-scheme $S$) a torsion sheaf of length $r$ on $\widehat{X}\times \C^g\times S$. But, as explained in \cite[Proposition 3.16]{FT}, length $r$ torsion sheaves on $\widehat{X}\times \C^g$ are corepresented by the symmetric product $\Sym^r(\widehat{X}\times\C^g)$. 
\end{proof}

\begin{remark}
Using the techniques developed in \cite[Section 1]{GKUW_semihomogeneous} one can prove a generalization of Proposition \ref{prop_modulispace} to describe a moduli space of semi-homogeneous Higgs bundles (in the sense of \cite{Mukai_semihomogeneous}) of fixed rank and Neron-Severi class $H$ on $X$ as a suitable symmetric power on a suitable abelian torsor $M_{H,1}(X)$ generalizing the dual abelian variety $\widehat{X}$. 
\end{remark}

%%%%%%%%%%%%%%%%%%%%%%%%%%%%%%%%%%%%%%%%%%%%

\section{The Hitchin fibration and its compactification} \label{section_Hitchin}

When $X$ is a compact Riemann surface, the Dolbeault moduli space is naturally endowed with a Lagrangian fibration, with an affine base and generically with complex tori as fibers, called the {\it Hitchin morphism}
\begin{equation}\label{eq_Hitchin}\begin{split}
 h_X\colon M ^r _{\mathrm{Dol}}(X) &\longrightarrow \bigoplus_{i=1}^r H^0\big(X, S^i\Omega^1_X\big)\\
 (E,\phi) &\longmapsto \chi_\phi
\end{split}\ .\end{equation}
This is given by associating to each Higgs bundle $(E,\phi)$ the coefficients of the characteristic polynomial $\chi_\phi$, where $S^i\Omega^1$ denotes the $i$-th symmetric power of cotangent bundle $\Omega_X^1$.

\begin{example}\label{example_ellipticcurve}
For an elliptic curve $X$, the canonical bundle $\Omega_X^1$ is trivial,  and   the moduli space $M_{\textrm{Dol}}^r(X)$ is isomorphic to $\Sym^r(X\times\C)$  by Proposition \ref{prop_modulispace}. The corresponding  Hitchin fibration 
\begin{equation*}
h_X\colon M_{\textrm{Dol}}^r(X)\cong \Sym^r(X\times \C)\longrightarrow \C^r\cong \bigoplus_{i=1}^r H^0(X,\Omega^1)
\end{equation*}
may be described via
\begin{equation*}
\bm{\lambda}\longmapsto \prod_{i=1}^r(x-\lambda_i)=x^r+\sigma_1(\bm{\lambda})x^{r-1}+\cdots + \sigma_{r-1}(\bm{\lambda})x + \sigma_{r}(\bm{\lambda})\ .
\end{equation*}
Here $\bm{\lambda}=(\lambda_1,\ldots, \lambda_n)\in\C^g$, and $\sigma_1,\ldots, \sigma_n$ are the $r$-th elementary symmetric polynomials (for $1\leq r\leq n$ and evaluated in $\bm{\lambda}$). Phrased geometrically, this means that the Hitchin fibration factors as
\begin{equation*}
h_X\colon M_{\textrm{Dol}}^r(X)\cong \Sym^r(X\times \C)\xlongrightarrow{\Sym^r(\pr_2)}\Sym^r\C\cong \C^r
\end{equation*}
where the isomorphism on the right is given by the $\sigma_i$. 
\end{example}

\subsection{A factorization of the Hitchin morphism} From now on let $X$ be a complex abelian variety of dimension $g\geq 1$, and write $\widehat{X}$ for its dual. One may define a Hitchin morphism 
\begin{equation*}
h_X\colon M_{\textrm{Dol}}^r(X)\longrightarrow \bigoplus_{i=1}^r H^0\big(X, S^i\Omega^1_X\big)\cong\bigoplus_{i=1}^r S^i\C^g\cong \mathbb{C}^{N}
\end{equation*}
exactly as in the curve case \eqref{eq_Hitchin}, where $N=\sum_{i=1}^r{{g+i-1}\choose {i}}$. This Hitchin morphism is, however, not surjective and thus cannot be a Lagrangian fibration.

The elliptic curve case in Example \ref{example_ellipticcurve} motivates the following.

\begin{proposition}\label{prop_Hitchinfactorization}
Let $X$ be a complex abelian variety of dimenson $g\geq 1$. Then the Hitchin morphism  naturally factors as
\begin{equation*}
M_{\mathrm{Dol}}^r(X)\xlongrightarrow{\sd_X}\Sym^r(\C^g) \hooklongrightarrow \C^{N} \ ,
\end{equation*}
where the second arrow is a closed immersion given by the evaluation
\begin{equation*}
\big[\bm{\lambda}_1, \ldots, \bm{\lambda}_r\big]\longmapsto \big[\sigma_i(\bm{\lambda}_1,\ldots, \bm{\lambda}_r)\big]_{i=1}^r\in\bigoplus_{i=1}^r S^i\C^g
\end{equation*}
of a collection of $r$ unordered vectors $\bm{\lambda}_1, \ldots, \bm{\lambda}_n\in\C^g$ in the $r$-th elementary symmetric polynomials $\sigma_i$. 
\end{proposition}

\begin{definition}
The morphism $\sd_X\colon M_{\mathrm{Dol}}^r(X)\rightarrow\Sym^r(\C^g)$ is called the \emph{spectral data morphism}.
\end{definition}

The name comes from the fact that one can see the factorization in Proposition \ref{prop_Hitchinfactorization} as a special case of the factorization constructed in \cite[Proposition 5.1]{CN_Hitchin}. We nevertheless give an independent and elementary proof of this fact.

\begin{proof}[Proof of Proposition \ref{prop_Hitchinfactorization}]
By Proposition \ref{prop_modulispace} we may identify $M_{\textrm{Dol}}^r(X)$ with $\Sym^r(\widehat{X}\times \C^g)$. Via this identification we can define $\sd_X$ as the composition 
\begin{equation*}
\sd_X\colon M_{\textrm{Dol}}^r(X)\cong \Sym^r\big(\widehat{X}\times \C^g\big)\xlongrightarrow{\Sym^r(\pr_2)}\Sym^r\C^g \ .
\end{equation*}
For $\big[\bm{\lambda}_1,\ldots, \bm{\lambda}_r\big]\in \Sym^r(\C^g)$ the coefficients of the characteristic polynomial $\chi(x)=(x-\bm{\lambda}_1)\cdots (x-\bm{\lambda}_r)$ are precisely the $\sigma_i(\bm{\lambda}_1,\ldots, \bm{\lambda}_r)\in S^i\C^g$ and so this defines a factorization of the Hitchin morphism. 
\end{proof}

The spectral data morphism has the following geometric properties.

\begin{proposition}\label{prop_Hitchin}
Let $X$ be a complex abelian variety of dimension $g$, and $\widehat{X}$ the dual abelian variety. 
\begin{enumerate}[(i)]
    \item The spectral data morphism is proper.
    \item Every fiber of the spectral data morphism is isomorphic to $\Sym^r\widehat{X}$.
    \item The Dolbeault moduli space $M_{\textrm{Dol}}^r(X)\cong \Sym^r(\widehat{X}\times\C^g)$ is naturally a symplectic orbifold and the spectral data morphism is an orbifold Lagrangian fibration. 
\end{enumerate}
\end{proposition}

\begin{proof}
By definition the spectral data morphism $\mathrm{Sym}^r(\mathbb{C}^g)$ fits into a commutative diagram

\[
\xymatrix{ (\widehat{X}\times \mathbb{C}^g)^r \ar[d]_{(\pr_2)^r}  \ar[r]^{ q_1 }& \mathrm{Sym}^r(\widehat{X} \times \mathbb{C}^g) \ar[d]^{\mathrm{Sym}^r(\pr_2) }\\ \mathbb{C}^{gr} \ar[r]^{q_2}& \mathrm{Sym}^r(\mathbb{C}^g) ,\\ 
}
\]
where the $q_i$ denote quotient maps by the $S_r$-operation. In part (i) the properness of $\Sym^r(\pr_2)$ follows from $(\pr_2)^r$ being  proper. 
Since $q_1$ and $q_2$ are surjective, we have
\begin{align*}
    \mathrm{Sym}^r(\pr_2)^{-1}\big([\bm{\lambda}_1,...,\bm{\lambda}_r]\big) & = q_1 \big( \pr_2 ^{-1} \big(q_2^{-1}\big([\bm{\lambda}_1,...,\bm{\lambda}_r]\big)\big)\big) \\ & = q_1 \big([\widehat{X}\times\{ \bm{\lambda}_1\},..., \widehat{X}\times \{\bm{\lambda}_r\}]\big) \\ & \cong \widehat{X}^r/S_r = \mathrm{Sym}^r(\widehat{X})
\end{align*}
for $\bm{\lambda}_i\in\C^g$ (and $i=1,\ldots r$) and this shows Part (ii). For Part (iii) it is enough to note that $T^\ast X\cong X\times \C^g$ naturally carries the structure of a symplectic manifold and that $\pr_2$ is a Lagrangian fibration. Passing to the quotients along $q_1$ and $q_2$ makes $\Sym^r(\widehat{X}\times \C^g)$ into a symplectic orbifold and  $\Sym^r(\pr_2)$ into an orbifold Lagrangian fibration.
\end{proof}

\subsection{Compactification of the spectral data morphism}\label{section_compactification}
In \cite{Hausel_compactification}, \cite[Section 11]{Simpson_Hodgefiltration}, and \cite[Theorem 3.1.1]{deCataldo_compactification}, the authors propose a compactification of both the Dolbeault moduli space and the base admitting an extension of the Hitchin fibration. The distinguishing feature of this compactification is that it is naturally compatible with a certain $\G_m$-operation. 

In our case we can construct such a compactification by elementary means:

\begin{proposition}\label{prop_compactification}
Let $X$ be a complex abelian variety of dimension $g\geq 1$ and write $\widehat{X}$ for its dual. There is a compactification $\Mbar_{\textrm{Dol}}^r(X)$ of $M_{\textrm{Dol}}^r(X)$ as well as an extension 
\begin{equation*}\overline{\sd}_X\colon \Mbar_{\textrm{Dol}}^r(X)\longrightarrow  \Sym^r(\P^g)
\end{equation*}
of the spectral data morphism that fits into the following cartesian diagram
\begin{equation*}\begin{tikzcd}
M_{\textrm{Dol}}^r(X)\arrow[rr]\arrow[d,"\sd_X"']&&\Mbar_{\textrm{Dol}}^r(X)\arrow[d,"\overline{\sd}_X"]\\
\Sym^r(\C^g)\arrow[rr]&&\Sym^r(\P^g)
\end{tikzcd}\ ,\end{equation*}
and fulfills the following properties:
\begin{enumerate}[(i)]
\item the variety $\Mbar_{\textrm{Dol}}^r(X)$ is proper over $\Spec\C$ and has at most orbifold singularities;
\item the horizontal arrows are open and dense immersions;
\item the extension $\overline{\sd}_X$ is a proper  morphism that defines an orbifold Lagrangian fibration with generic fibers isomorphic to $\Sym^r\widehat{X}$; and
\item the $\G_m$-operation on $M_{\textrm{Dol}}^r(X)$ given by rescaling the Higgs field extends to an operation on all of $\Mbar_{\textrm{Dol}}^r(X)$ and makes the spectral data morphism $\G_m$-equivariant,  where $\Sym^r(\C^g)$ and $\Sym^r(\P^g)$ are endowed with $\G_m$-action induced from the diagonal operation on $\C^g\subseteq\P^g$. 
\end{enumerate}
\end{proposition}

\begin{proof}
Define the compactification $\Mbar_{\textrm{Dol}}^r(X)$ of $M_{\textrm{Dol}}^r(X)$ by 
\begin{equation*}
    \Mbar_{\textrm{Dol}}^r(X)=\Sym^r(\widehat{X}\times \P^g)
\end{equation*}
where $M_{\textrm{Dol}}^r(X)$ may be identified with an open subset of $\Mbar_{\textrm{Dol}}^r(X)$ via the inclusion induced by $\mathbb{C}^g\subseteq\P^g$. We define the compactified spectral data morphism again as the composition
\begin{equation*}
\overline{\sd}_X\colon \Mbar_{\textrm{Dol}}^r(X)=\Sym^r(\widehat{X}\times \P^g)\xlongrightarrow{\Sym^r(\pr_2)}\Sym^r(\P^g) \ .
\end{equation*}

This definition immediately implies Part (i) and (ii) as well as that the diagram commutes and is cartesian. Part (iii) follows from the fact that $\overline{\sd}_X$ is a finite group quotient of the projection $\pr_2\colon \Sym^r(\widehat{X}\times\P^r)\rightarrow \Sym^r(\P^g)$ as in the proof of Proposition \ref{prop_Hitchin}.  

For Part (iv): The $\G_m$-operation on $M_{\textrm{Dol}}^r(X)$ is induced by the diagonal operation on $\C^g$ under the isomorphism $M_{\textrm{Dol}}^r(X)\cong \Sym^r(\widehat{X}\times \C^g)$. The diagonal operation of $\G_m$ factors through the operation of the dense torus $T=\G_m^{n+1}/\G_m$ on $\C^g$ and, thus, naturally extends to an operation on $\P^g$.
Therefore the $\G_m$-operation extends to an operation on $\Mbar_{\mathrm{Dol}}^r(X)=\Sym^r(\widehat{X}\times\P^g)$ and makes  $\overline{\sd}_X$ into a $\G_m$-equivariant morphism. 
\end{proof}

\section{The non-abelian Hodge correspondence}\label{section_nonabelianHodge}

Let $X$ be a complex abelian variety of dimension $g$. Then we may write $X$ as $\C^g/L$ for a lattice $L$ of full rank $2g$. The fundamental group of $X$ is naturally isomorphic to $L$. The non-abelian Hodge correspondence is a condensation of \cite{Corlette, Donaldson, Hitchin, Simpson_nonabelianHodge}  and tells us that there is a natural one-to-one correspondence between semisimple representations $\rho\colon L\rightarrow \GL_r(\mathbb{C})$ and topologically trivial polystable Higgs bundles $(E,\phi)$ of rank $r$ on $X$. 

Write $M^r_{\mathrm{Betti}}(X)$ for the character variety of $L$, also known as the \emph{Betti moduli space}. It is given as the geometric invariant theory quotient
\begin{equation*}
    \Hom\big(L,\GL_r(\mathbb{C})\big)\sslash_{\GL_r({\C})}
\end{equation*}
where $\GL_r(\C)$ acts by conjugation. The non-abelian Hodge correspondence provides us with a real-analytic isomorphism
\begin{equation}\label{eq_generalnonabelianHodge}
\eta_X^r\colon M_{\mathrm{Dol}}^r(X)\xlongrightarrow{\sim} M_{\textrm{Betti}}^r(X) 
\end{equation}
that sends a semisimple representation to the class of the associated polystable Higgs bundle (see \cite[Theorem 7.18 and Proposition 7.8]{SimII}).
Using the structure of $X$ as an abelian variety, we will now make $\eta_X^r$ explicit.

\subsection{The non-abelian Hodge correspondence in rank one}
When $r=1$ we have a one-to-one correspondence between representations $\rho\colon L\rightarrow\mathbb{C}^\ast$ and topologically trivial Higgs bundles of rank one. In order to describe the induced isomorphism morphism 
\begin{equation*}
\eta_X^1\colon M_{\textrm{Betti}}^1(X)\xlongrightarrow{\sim} M_{\textrm{Dol}}^1(X)
\end{equation*}
we note that $M_{\textrm{Betti}}^1(X)=\Hom(L,\C^\ast)$ and $M_{\textrm{Dol}}^1(X)=\widehat{X}\times H^{1,0}(X)$ and consider the following factorization:
\begin{equation*}
\underbrace{\Hom(L,\C^\ast)}_{M_{\textrm{Betti}}^1(X)}\cong \Hom(L,S^1)\times \Hom(L,\R)\xlongrightarrow{\sim} \Pic_0(X)\times H^1(X,\R) \cong \underbrace{\widehat{X}\times H^{1,0}(X)}_{M_{\textrm{Dol}}^1(X)} \ . 
\end{equation*}

The first isomorphism is induced by the (real analytic) polar decomposition $\C^\ast\cong S^1\times\R$ given by $z\mapsto \big(\frac{z}{\vert z\vert},-\log\vert z\vert\big)$.
In the first factor, the second isomorphism is given by associating to a representation $\phi\colon L\rightarrow U(1)$ the line bundle $L(\phi)=(V\times \mathbb{C})/L$ on $X$, where $l\in L$ acts on $V\times \mathbb{C}$ via 
\begin{equation*}
(x,v)\longmapsto \big(x+l,\phi(l)v\big) \ .
\end{equation*}
In the second factor, the third isomorphism is given by writing a one-form $\phi$ in $H^1(X,\R)$ as $\phi=\varphi+\overline{\varphi}$ for a unique holomorphic $1$-form $\varphi\in H^{1,0}(X)$ and thus a Higgs field. This identifies $H^1(X,\R)$ with $H^{1,0}(X)$.

\subsection{The non-abelian Hodge correspondence for abelian varieties - made explicit}\label{section_explicitNAH} 

Once we choose a basis of the lattice $L$, a point in the space $\Hom\big(L,\GL_r(\mathbb{C})\big)$ is nothing but a tuple of $2g$ invertible pairwise commuting matrices. The GIT-stable locus in  $\Hom\big(L,\GL_r(\mathbb{C})\big)$ consists of those tuples of invertible matrices that are simultaneously diagonalizable. Therefore we find
\begin{equation*}
 M_{\textrm{Betti}}^r(X)\cong \Sym^r \big( ( \mathbb{C}^\ast)^{2g} \big)\ .
\end{equation*}

In terms of moduli spaces the non-abelian Hodge correspondence and the spectral data morphism from Section \ref{section_Hitchin} can be made explicit as follows.

\begin{proposition}\label{prop_nonAbelianHodge}
The non-abelian Hodge correspondence induces a real-analytic isomorphism 
\begin{equation*}\eta_X^r\colon M^r_{\mathrm{Dol}}(X)\xlongrightarrow{\sim}M^r_{\mathrm{Betti}}(X)
\end{equation*}
that makes the diagram
\begin{equation}\label{eq_nonabelianHodge}\begin{tikzcd}
M^r_{\mathrm{Dol}}(X)\arrow[dr,"\sim"] \arrow[dddr,bend right=,swap,"\sd_X"]\arrow[rrr,"\sim","\eta_X^r"']&&&M^r_{\mathrm{Betti}}(X)\arrow[d,"\sim"]\\
&\Sym^r(\widehat{X}\times \C^g)\arrow[dd,"\Sym^r(\pr_2)"']\arrow[rr,"\sim","\Sym^r (\eta_X^1)"']& & \Sym^r\big((\C^\ast)^{2g}\big)\arrow[dd,"\Sym^r\big(-\log\vert.\vert\big)"]\\\\
& \Sym^r(\mathbb{C}^g)\arrow[rr,"\sim"]&&\Sym^r\big(\mathbb{R}^{2g}\big)
\end{tikzcd}\end{equation}
commute.  
\end{proposition}

\begin{proof}
A representation $\pi_1(X)\rightarrow \GL_r(\C)$  is semisimple if and only if it is a direct sum of irreducible representations. Since $L=\pi_1(X)$ is abelian, Schur's Lemma implies that every irreducible representation of $L$ is one-dimensional. 

On the other side, a topologically trivial polystable Higgs bundle on $X$ is a direct sum of topologically trivial stable Higgs bundles on $X$. We have seen in Proposition \ref{prop_stableHiggs} above that the classification of homogeneous bundles on $X$ implies that topologically trivial stable Higgs bundles on $X$ are of rank one. 

Since the non-abelian Hodge correspondence is compatible with direct sum decompositions, this shows that the upper square commutes. Since $\eta_X^1$ is a real analytic isomorphism, so is $\eta_X^r$.

 The lower right vertical arrow is a symmetric product $\Sym^r(.)$ of 
\begin{equation*}\begin{split}
-\log\vert .\vert \colon (\C^\ast)^{2g}&\longrightarrow \R^{2g}\\
\big[z_i\big]_{i=1}^{2g}&\longmapsto \big[-\log\vert z_i\vert\big]_{i=1}^{2g} \ .
\end{split}\end{equation*}
The commutativity of the lower square follows from the fact that in rank one $\eta_X^1$ is induced by the polar decomposition $\C^\ast\cong S^1\times \R$ given by $z\mapsto \big(\frac{z}{\vert z\vert},-\log\vert z\vert\big)$. 

The triangle on the left commutes by definition.
\end{proof}

%%%%%%%%%%%%%%%%%%%%%%%%%%%%%%%%%%%%%%%%%%%%%%%%%%%%%%%%%%%%%%%%%%%%%%%%%%%%%%%%%

\section{A geometric $P=W$ phenomenon}

In this section we recall the notion of dual complexes of  divisorial log terminal compactifications of algebraic varieties as introduced in \cite{dFKX_dualcomplexes}, and use this language to prove Theorem \ref{mainthm_geometricP=W} from the introduction. 

\subsection{Dual complexes of dlt compactifications} Let $(Z,D)$ be a pair consisting of a smooth variety $Z$ and a divisor $D=D_{1}+\ldots+D_n$ with simple normal crossings (SNC). Then 
\begin{itemize}
\item every $D_i$ is smooth and
\item for $I\subseteq \{1,\ldots, n\}$ the  components of $\bigcap_{i\in I}D_i$ are irreducible and of codimension $\vert I\vert$ in $X$. 
\end{itemize}
We may associate to $D$ a $\Delta$-complex  $\Delta(X,D)$, called the \emph{dual complex} (as e.g. in \cite{dFKX_dualcomplexes, Payne_boundarycomplexes, Stepanov, Thuillier}). It consists of a $k$-simplex for every component of $\bigcap_{i\in I}D_i$ with $\vert I\vert=k$; a simplex is a face of another simplex if the corresponding components are contained in each other. Recall that by the results of \cite{Stepanov, Thuillier, Payne_boundarycomplexes}, the (simple) homotopy type of the dual complex only depends on the open locus $Z-D$.

In \cite{dFKX_dualcomplexes} the authors noted that one may also associate a dual complex $\Delta(Z,E)$ to a dlt pair $(Z,E)$ by considering only the divisor $D=E^{=1}$
of components whose coefficient in $E$ are $=1$, and mimicking the above construction. Let $Z^{\textrm{SNC}}$ be the open sublocus on which $E\vert_{Z^{\textrm{SNC}}}$ has simple normal crossings. Then we have $\Delta(Z,E)=\Delta\big(Z^{\textrm{SNC}},E\vert_{Z^{\textrm{SNC}}}\big)$. 

Suppose now that we are given two dlt pairs $(Z_i,E_i)$ (for $i=1,2$) that are \emph{crepant birational}, i.e. for which there are proper birational map $Z\rightarrow Z_i$ such that the pullbacks of $E_1$ and $E_2$ are equal. Then \cite[Proposition 11]{dFKX_dualcomplexes} tells us that $\Delta(Z_1,E_1)$ und $\Delta(Z_2,E_2)$ are even piecewise linear homeomorphic. 

Assume we are given a variety $U$ and a compactification $Z$ of $U$ such that $(Z,E)$ with $E=Z-U$ is a dlt pair and at the same time log Calabi-Yau, meaning that $K_X+E\sim_\Q 0$. Then the dual complexes of all dlt and log CY compactifications that are crepant birational to $(Z,E)$ are piecewise linear homeomorphic. So, by a slight abuse of notation, it makes sense to denote by $\Delta(U)$ the piecewise linear homeomorphism class of these dual complexes. We will refer to $\Delta(U)$ simply as the \emph{dual complex} of $U$. 

\subsection{A geometric $P=W$ phenomenon - Proof of Theorem \ref{mainthm_geometricP=W}}
Let $X$ be a complex abelian variety of dimension $g\geq 1$. Write $\widehat{X}$ for its dual, as well as 
\begin{equation*}
\sd_X\colon M_{\textrm{Dol}}^r(X)\cong\Sym^r(\widehat{X}\times \C^g)\xrightarrow{\Sym^r(\pr_2)} \Sym^r\C^g
\end{equation*}
for the spectral data morphism in the sense of Section \ref{section_compactification}.

Consider $\mathbb{B}^\ast := \Sym^r(\C^g)-\{0\}$. This so-called \emph{semi-algebraic deleted neighborhood} of $\{0\}$ admits a natural deformation retraction along the rays through zero onto the quotient $S^{2gr-1}/S_r$ of the sphere $S^{2gr-1}\subseteq \mathbb{C}^{gr}$. The preimage $N_{\textrm{Dol}}^r(X)$ of $\mathbb{B}^\ast$ in $M^r_{\textrm{Dol}}(X)$ forms a semi-algebraic open neighborhood of the  compactification $\Mbar_{\textrm{Dol}}^r(X)$ of $M_{\textrm{Dol}}^r(X)$ introduced in Proposition \ref{prop_compactification}. 
Composing the spectral data morphism with the retraction map along rays through the origin provides us with a continuous map 
\begin{equation*}
\widetilde{\sd}_X\colon N_{\textrm{Dol}}^r(X)\xlongrightarrow{\sd_X} \mathbb{B}^\ast\longrightarrow S^{2gr-1}/S_r \ .
\end{equation*}

We  want to identify the sphere quotient $S^{2gr-1}/S_r$ with the dual complex $\Delta\big(M_{\mathrm{Betti}}^r(X)\big)$ of a certain dlt compactification of $M_{\mathrm{Betti}}^r(X)$. Write $T=\mathbb{C}^{2g}$, and consider any smooth toric compactification $Z$ of $T$, e.g.\ $Z=\big(\PP^1\big)^{2g}$, given by a rational polyhedral fan $\Sigma_Z$ in the vector space $\R^{2g}$ generated by the cocharacter lattice of $T$ whose support is all of $\R^{2g}$. We may identify the dual complex of the toric variety $Z^r$ with respect to the toric boundary $Z^r-T^r$ with the link of $\Sigma_{Z^r}=\Sigma_Z^r$ in $\R^{2gr}$ and this endows $S^{2gr-1}$ with a piecewise linear structure. 

By \cite[Proposition 5.20]{KollarMori} the symmetric product $\Sym^r(Z^g)$ defines a log canonical log Calabi Yau compactification of $M_{\textrm{Betti}}^r(X)=\Sym^rT$. We may now apply the algorithm in \cite[Lemma 4.3.1]{MMS_geometricP=W} to find that $M_{\textrm{Betti}}^r(X)$ admits a dlt log Calabi-Yau compactification that is crepant birational to $\big(\Sym^r(Z),D^{(r)}\big)$ and has the same dual complex by \cite[Proposition 11]{dFKX_dualcomplexes}.

The operation of $S_r$ on $Z^r$ is free in codimension one. Thus, we may apply \cite[Theorem 3.0.7]{MMS_geometricP=W} and find that the dual complex of $\Sym^r(Z)$ with respect to the quotient of the toric boundary compactification is piecewise linearly homeomorphic to $S^{2gr-1}/S_r$.

To complete the picture we define $N_{\textrm{Betti}}^r(X)$ to be the preimage of $\mathbb{B}\subseteq \Sym^r\big(\R^{2g}\big)$ under $-\log\vert.\vert$ and 
\begin{equation*}
\rho\colon N_{\textrm{Betti}}^r(X)\longrightarrow \Delta\big(M_{\mathrm{Betti}}^r(X)\big)\cong S^{2gr-1}/S_r
\end{equation*}
to be the composition of $\Sym^r(-\log\vert.\vert)$ with the natural retraction along rays through zero modulo $S_r$. Finally, the commutative diagram \eqref{eq_nonabelianHodge} in Proposition \ref{prop_nonAbelianHodge} implies that $\eta_X^r$ induces a homeomorphism between $N_{\textrm{Dol}}^r(X)$ and $N_{\textrm{Betti}}^r(X)$ and that
\begin{equation*}\begin{tikzcd}
 N^\ast_{\mathrm{Dol}}(X)\arrow[d,"\widetilde{\sd}_X"']\arrow[rr,"\sim","\eta_X^r"']&&N^\ast_{\mathrm{Betti}}(X)\arrow[d,"\rho"]\\
 S^{2gr-1}/S_r\arrow[rr,"\sim"]&&\Delta\big(M_{\mathrm{Betti}}^r(X)\big)
\end{tikzcd}\end{equation*}
also commutes. 

\begin{remark}\label{remark_notasphere}
The sphere quotient $S^{2gr-1}/S_r$ has the rational homology of a sphere, since $S_r$ acts with determinant one. It is known that $S^{2gr-1}/S_r$ is piecewise linear homeomorphic to a sphere, when $r=1$ or $g=1$ (see \cite[Theorem 1.5.3]{MMS_geometricP=W} for a proof of the latter). But, as soon as $g\geq 2 $ and $r\geq 2$, the quotient $S^{2gr-1}/S_r$ is itself not even a topological manifold and, in particular, not piecewise linear homeomorphic to a sphere.

This can be seen as follows: Let $p\in S^{2gr-1}$ be a point at which the $S_r$-operation has stabilizer $\Z/2\Z$ (which exists, since $r\geq 2$). Write $B^i$ for the open unit ball of dimension $i$ and set $N=2gr-1$ as well as $k=N-2g$. Then, locally around $p$, the quotient may be described as the quotient of $B^{k}\times B^{N-k}$ by the involution $\id_{B^k}\times (-\id_{B^{N-k}})$. So, locally around $p$, the quotient $S^{2gr-1}/S_r$ is homeomorphic to $Z:=B^k\times C\big(\R\P^{N-k-1}\big)$ with $p$ corresponding to the tip $q$ of the cone $C\big(\R\P^{N-k-1}\big)$ over $\R\P^{N-k-1}$. 
Recall that we have 
\begin{equation*}
H_s\big(B^k,B^k-\{0\}\big)\cong \widetilde{H}_{s-1}\big(S^{k-1}\big)=
\begin{cases} \Z, & \textrm{ if } s=k\\
0, & \textrm{ else,} 
\end{cases}
\end{equation*}
as well as 
\begin{equation*}
H_t\big(C(\R\P^{N-k-1}),C(\R\P^{N-k-1})-\{q\}\big)\cong\widetilde{H}_{t-1}\big(\R\P^{N-k-1}\big)=\begin{cases}\Z, & \textrm{ if } t=N-k\\
\Z/2\Z, & \textrm{ if } 1< t< N-k \textrm{ even}\\
0, & \textrm{ else.}
\end{cases}
\end{equation*}
Thus, on the one hand, we may use K\"unneth's formula for pairs (and the vanishing of all $\Tor$'s) to find:
\begin{equation}\label{eq_nomanifold}\begin{split}
H_i\big(S^N/S_k,S^N/S_k-\{p\}\big)&\cong H_i\big(Z,Z-\{p\}\big)\\
&\cong \bigoplus_{s+t=i}H_s\big(B^k,B^k-\{0\}\big)\otimes H_t\big(C(\R\P^{N-k-1}),C(\R\P^{N-k-1})-\{q\}\big)\\&= \begin{cases}\Z, & \textrm{ if } i=N\\
\Z/2\Z, & \textrm{ if } k+1< i< N  \textrm{ odd}\\
0, & \textrm{ else.}
\end{cases}
\end{split}\end{equation}
On the other hand, for a point $p$ on an $N$-dimensional topological manifold $M$ we have: 
\begin{equation}\label{eq_manifold}
H_i\big(M, M-\{p\}\big)\cong H_i\big(B^N,B^N-\{0\}\big)\cong\widetilde{H}_{i-1}(S^{N-1})=\begin{cases} \Z, & \textrm{ if } i=N\\
0, & \textrm{ else.} 
\end{cases}
\end{equation}
As soon as $g\geq 2$, we have $N-(k+1)=2g-1>2$ and so there is an odd index $i$ with $k+1<i<N$. In this case, if $S^N/S_r$ were a manifold, the formulas \eqref{eq_nomanifold} and \eqref{eq_manifold} would constitute a contradiction.
\end{remark}

\section{The perverse filtration}\label{section_perversefiltration}

We devote this section to  the computation of the perverse filtration associated to the Hitchin fibration in the case of $X$ a complex abelian variety. For background material the reader is invited to look at  \cite{GM}, \cite{GelM}, and \cite{DC}.

Let $X$ be a non-singular algebraic variety of dimension $n$ over $\mathbb{C}$. We work in the bounded derived category $\mathcal{D}(X):=\mathcal{D}^b_c(X)$ of constructible complexes on $X$. We write $\omega_X := c^!\mathbb{Q}_X$ for the \emph{dualizing complex}, where $c: X \to \{ pt \}$ is the constant map. This gives rise to a functor
\[ 
\mathbb{D} : \mathcal{D}(X) \longrightarrow \mathcal{D}(X), \hspace{0.5cm} E \longmapsto    \bfR\Hom(E, \omega _X)\ .
\]
The object $\mathbb{D}E$ is the {\it Grothendieck-Verdier dual} of $E$ (see e.g.\ \cite[Section 5.16]{GelM} for further background and more details).

\subsection{The perverse $t$-structure}
As explained e.g. in \cite[7.1.2 and 7.1.7]{GelM}, there is a distinguished t-structure on the triangulated category $\mathcal{D}(X)$, the \emph{perverse} t-structure $(^{\mathfrak{p}}\mathcal{D}^{\leq 0}, ^{\mathfrak{p}}\mathcal{D}^{\geq 0})$ defined as follows:
\begin{equation}\label{eq_pervtstructure}\begin{split}
^{\mathfrak{p}}\mathcal{D}^{\leq 0} &= \{ E \in \mathcal{D}^b_c(X) \ | \ \mathrm{dim}\ \mathrm{Supp}(\mathcal{H}^i(E) ) \leq i \textrm{ for all } i\in \mathbb{Z}  \}  \\
^{\mathfrak{p}}\mathcal{D}^{\geq 0} & = \{ E \in \mathcal{D}^b_c(X) \ | \ \mathrm{dim}\ \mathrm{Supp}(\mathcal{H}^i(\mathbb{D}E) ) \leq i \textrm{ for all } i\in \mathbb{Z}  \} 
\end{split}\end{equation}
The associated truncation functors are
\begin{equation*} 
^{\mathfrak{p}}\tau_{\leq 0}: \ ^\mathfrak{p}\mathcal{D}^{\leq 0} \longrightarrow \mathcal{D}(X) \qquad \textrm{and} \qquad
^{\mathfrak{p}}\tau_{\geq 0}: \ ^\mathfrak{p}\mathcal{D}^{\geq 0} \longrightarrow \mathcal{D}(X)\ .
\end{equation*}
As usual, we write  $^{\mathfrak{p}}\tau_{\leq k} (E) := \big( {}^{\mathfrak{p}}\tau_{\leq 0}(E[k])\big)[-k]$ and, respectively, $^{\mathfrak{p}}\tau_{\geq k} (E) := \big({}^{\mathfrak{p}}\tau_{\geq 0}(E[k])\big)[-k]$. We recall that  truncation functors commute with taking shifts, i.e.\ we have
\begin{equation} \label{eq_shifts}
^{\mathfrak{p}}\tau_{\leq k} \circ [i] = [i] \circ{}^{\mathfrak{p}}\tau_{\leq k+i}
\end{equation}
for each $i\in \mathbb{Z}$  (see e.g.\  \cite[Section 4.1]{IntersectionFT}).

It is traditional to denote the heart of the perverse t-structure by
\[
\mathrm{Perv}(X):= {^{\mathfrak{p}}\mathcal{D}^{\leq 0} \cap {^{\mathfrak{p}}\mathcal{D}^{\geq 0}}} \ ,
\]
and to refer to its objects as {\it perverse sheaves}. Specifically, perverse sheaves  are complexes of constructible sheaves satisfying the {\it support} and {\it co-support condition} appearing in  \eqref{eq_pervtstructure}.

\subsection{The decomposition theorem} \label{section_decomposition} Let $f\colon X\rightarrow Y$ be a proper morphism of algebraic varieties with at most finite quotient singularities and whose fibers also have at most finite quotient singularities. Let $n=\dim X$, and write \begin{equation*}
d=\dim X\times_Y X - \dim X\end{equation*}
for the so-called \emph{semi-small defect} of $f$.

In this situation the decomposition theorem of Bernstein--Beilinson--Deligne \cite{BBD}  (see e.g.\ \cite[Section 9]{Maxim_intersection&perverse})
%(see \cite{DecoTHM} Example 2.5.3 for this version)
tells us the following.

\begin{enumerate}[(i)]
\item There is a direct sum decomposition 
\[ 
\bfR f_*\mathbb{Q}_X = \bigoplus _{j=0}^{2d}R^jf_*\mathbb{Q}_X[-j] 
\]
in $\mathcal{D}(Y)$. 
\item We have
\[ ^{\mathfrak{p}}\mathcal{H}^j\big(\bfR f_*\mathbb{Q}_X[n]\big) = R^{d+j}f_*\big(\mathbb{Q}_X[\mathrm{dim}Y]\big) \]
%eliminate this, correct index, define H
for all $\ j\in \mathbb{Z}$.
\item Grothendieck-Verdier duality gives us
\[ \mathbb{D}\big({}^{\mathfrak{p}}\mathcal{H}^j(\bfR f_*\mathbb{Q}_X[n])\big) = {}^{\mathfrak{p}}\mathcal{H}^{-j}\big(\bfR f_*\mathbb{Q}_X[n]\big). \]
\item If moreover $f$ is projective and $L\in H^2(X)$ is relatively ample, then the $k$-fold application of $L$ induces an isomorphism 
\begin{equation}\label{eq_relativehardLefschetz} L^j:{} ^{\mathfrak{p}}\mathcal{H}^j\big(\bfR f_*\mathbb{Q}_X[n]\big) \xlongrightarrow{\sim} {} ^{\mathfrak{p}}\mathcal{H}^{-j}\big(\bfR f_*\mathbb{Q}_X[n]\big).\end{equation}
\end{enumerate}

Note that the splitting $\bfR f_*\mathbb{Q}_X [n]= \bigoplus _{j=-d}^{d} {}^{\mathfrak{p}}\mathcal{H}^j\big(\bfR f_*\mathbb{Q}_X[n]\big)$ is symmetric with respect to $0\in \mathbb{Z}$. The perverse cohomology complexes $^{\mathfrak{p}}\mathcal{H}^j\big(\bfR f_*\mathbb{Q}_X[n]\big)$ vanish outside the interval $[-d,d]$.

\begin{example}\label{example_constantmap}
For the constant morphism $f: X \to \text{Spec}(\mathbb{C})$ we have $ \bfR f_*\mathbb{Q}_X = H^*(X, \mathbb{Q})  \ , $
and hence
\[ 
{} ^{\mathfrak{p}}\mathcal{H}^j\big(\bfR f_*\mathbb{Q}_X[n]\big) = R^{n+j}f_*(\mathbb{Q}) = \begin{cases} H^{n+j}(X,\mathbb{Q}) & \textrm{ if } -n\leq j\leq n \\
0 & \textrm{ otherwise. }
\end{cases}
\]
\end{example}

\begin{remark}
The version of the decomposition theorem above is usually only stated this way when $X$ and $Y$ are smooth. In general, when $X$ or $Y$ are singular, the decomposition theorem remains true when replacing $\Q[n]$ and $\Q[\dim Y]$ by the intersection complex on $X$ and $Y$ respectively. 

If $X$ has at most finite quotient singularities, however, it is not necessary to work with intersection cohomology. Indeed, \cite[Proposition A1, iii)]{Bri} tells us that if $X$ has finite quotient singularities then it is rationally smooth. Therefore, \cite[Proposition 8.2.21]{THT}, where we replace complex coefficients with rational coefficients, allows us to conclude that the intersection complex $IC_X$ is indeed isomorphic to $\mathbb{Q}_X[\dim X]$ like in the smooth case. Hence the version of the decomposition theorem we recalled here remains intact even in the case when $X$ is non-smooth but has at most finite quotient singularities. 
\end{remark}

\subsection{The perverse filtration} We now briefly recall the construction of the perverse filtration. Let $f\colon X\rightarrow Y$ be a flat and projective morphism of algebraic varieties with at most finite quotient singularities and whose fibers also have at most finite quotient singularities.  Again let $n=\dim X$ and $d=\dim X\times_Y X - \dim X$. Given a complex $K^{\bullet}\in \mathcal{D}^b_c(Y) $, there is a natural morphism
\begin{equation*}
 ^{\mathfrak{p}}\tau _{\leq k}( K^{\bullet}) \longrightarrow  K^{\bullet}\ ,
\end{equation*}
which induces  a map on the level of hypercohomology:
\begin{align*}\label{hypercosimpl}
\bbH^{i}\big(Y, {}^{\mathfrak{p}}\tau _{\leq k}(K^{\bullet})\big) \longrightarrow \bbH^{i}(Y, K^{\bullet})\ .
\end{align*}
We consider the complex $K^{\bullet}$ to be the pushforward of the (shifted) intersection cohomology complex, which in our case is $\mathbf{R}f_* IC_X[-d] = \mathbf{R}f_*\mathbb{Q}_X[n-d]$
and obtain a map

\begin{equation}\label{hyperco}
    \mathbb{H}^{i-(n-d)} \big( Y, ^{\mathfrak{p}}\tau _{\leq k}(\mathbf{R}f_*\mathbb{Q}_X[n-d])\big) \longrightarrow \mathbb{H}^{i-(n-d)}\big(Y,\mathbf{R}f_\ast\Q_X[n-d]\big)\cong H^i(X, \mathbb{Q})
\end{equation}
where the last isomorphism comes from 
\begin{equation*}
\mathbb{H}^{i-(n-d)}\big(Y,\mathbf{R}f_\ast\Q_X[n-d]\big)\cong \mathbb{H}^i(Y, \mathbf{R}f_*\mathbb{Q}_X) \cong H^i(X, \mathbb{Q}) \ .
\end{equation*}

\begin{definition}[Perverse filtration]
The image of \eqref{hyperco} in the cohomology $H^i(X, \mathbb{Q})$ is denoted by 
\begin{equation*}
P^f_k\big(H^i(X, \mathbb{Q})\big) \subseteq H^i(X, \mathbb{Q})\ ,
\end{equation*} 
and defines a filtration indexed by $k$ on $H^i(X, \mathbb{Q})$, called the {\it perverse filtration} associated to $f$. 
\end{definition}

For more details on the construction of the perverse filtration we refer the reader e.g.\ to  \cite[Section 1.2 and 1.3]{Cataldo2019HitchinFA}). Since by the decomposition theorem the complex $\mathbf{R}f_*\mathbb{Q}[n]$ decomposes into a direct sum of its perverse cohomologies, which vanish outside the interval $[-d,d]$, we have $P_p^f \big(H^i(X,\Q)\big)=0$ for $p\leq -1$ and $P_p^f\big(H^i(X,\Q)\big)=H^i(X,\Q)$ for $p\geq 2d$. Notice that the shift makes sure that the perverse filtration is concentrated in degrees $[0,2d]$ rather than in $[-d,d]$.

\subsection{The perverse filtration for  abelian varieties} Let  $X$ be a complex abelian variety of dimension $g$ and write $\widehat{X}$ for its dual abelian variety. In this section we compute the perverse filtration associated to the spectral data morphism described in Section \ref{section_Hitchin}. 

\begin{proposition}\label{thm_Pfiltration}
The perverse filtration associated with the spectral data morphism
\begin{equation*}
\sd_X =  \mathrm{Sym}^r(\pr_2) : M_{\textrm{Dol}}^r(X)=\mathrm{Sym}^r(\widehat{X}\times  \mathbb{C}^g) \longrightarrow\mathrm{Sym}^r(\mathbb{C}^g)
\end{equation*}
described in Section \ref{section_Hitchin} is given by
\begin{equation}\label{perversekunneth}\begin{split}
P_k^{\sd}H^*\big(\mathrm{Sym}^r(\widehat{X}\times  \mathbb{C}^g), \mathbb{Q}\big) & = P_k^{\sd}H^*(\widehat{X}^r, \mathbb{Q}) ^{S_r} \ . %\\
\end{split}\end{equation}
\end{proposition}

\begin{proof}
By \cite[Proposition 2.14]{ZHA}, we have that  
\[ 
P_k^{\sd}H^*\big(M^r_{Dol}(X), \mathbb{Q}\big) = P_k^{\sd}H^*\big( \mathrm{Sym}^r(\widehat{X}\times  \mathbb{C}^g), \mathbb{Q} \big) \cong P_k^{\sd}H^*\big((\widehat{X}\times \mathbb{C}^g)^r, \mathbb{Q}\big)^{S_r} \ .
\]
and the homotopy invariance of singular cohomology yields
\begin{equation*} 
P_k^{\sd}H^*\big((\widehat{X}\times \mathbb{C}^g)^r, \mathbb{Q}\big) \cong P_k^{\sd}H^*\big(\widehat{X}^r\times \mathbb{C}^{gr}, \mathbb{Q}\big) \cong P_k^{\sd}H^*(\widehat{X}^r, \mathbb{Q}) 
\end{equation*}
since $\C^{gr}$ is contractible.
\end{proof}

By \cite[Corollary 2.2]{ZHA} we may rewrite \eqref{perversekunneth} as 
\begin{equation*}
P_k^{\sd} H^\ast\big(\Sym^r(\widehat{X}\times \C^g),\Q\big)=\big(\mathrm{Span}\{\alpha_1 \otimes...\otimes \alpha_r \ \text{s.t.} \ \mathfrak{p}(\alpha_1) + ...+ \mathfrak{p}(\alpha_r) \leq k\}\big)^{S_r} \label{perversekunneth}
\end{equation*}
where $\alpha_i\in H^*(\widehat{X},\mathbb{Q})$ and $\mathfrak{p}$ denotes the perversity 
of $\alpha_i$ (in this case the degree in the cohomology of $\widehat{X}$) for $i=1,...,r$.

\begin{proposition}
\label{prop_PfiltrationRankone}
If $r=1$ then the perverse filtration is trivial, that is,  we have
\begin{equation*}
P_kH^i\big(M_{\textrm{Dol}}^1(X),\Q\big)=\begin{cases} 0 & \textrm{ if } k\leq i-1 \\
H^i(\widehat{X},\Q) & \textrm{ if } k\geq i
\end{cases}
\end{equation*}
for all indices $k,i\in\Z$
\end{proposition}

\begin{proof}
This follows from Example \ref{example_constantmap}. Indeed, if $r=1$ then the spectral data morphism (as well as the Hitchin morphism) is the projection $\pr_2\colon\widehat{X}\times \mathbb{C}^g \to \mathbb{C}^g$, hence the defect of semi-smallness $d=\dim \widehat{X}$ is maximal and the perverse filtration is induced by
\begin{align}\label{trivfiltmap}
\mathbb{H}^i\big(\mathbb{C}^g, {}^{\mathfrak{p}}\tau_{\leq k}(\mathbf{R}(\pr_2)_*\mathbb{Q})\big) \longrightarrow \mathbb{H}^i(\mathbb{C}^g, \mathbf{R}(\pr_2)_*\mathbb{Q})\ .
\end{align}
Since
\begin{equation*}
\mathbb{H}^i\big(\mathbb{C}^g, {}^{\mathfrak{p}}\tau _{\leq k}(\mathbf{R}(\pr_2)_*\mathbb{Q})\big)= \begin{cases} 0 & \textrm{ if } k\leq i-1 \\
H^i(\widehat{X}\times \mathbb{C}^g,\Q) & \textrm{ if } k\geq i
\end{cases}
\end{equation*}
and 
\begin{equation*} \mathbb{H}^i(\mathbb{C}^g, \mathbf{R}(\pr_2)_*\mathbb{Q}) \cong H^i(\widehat{X}\times \mathbb{C}^g,\Q)\end{equation*}
the map (\ref{trivfiltmap}) is zero for $k\leq i-1$ and the identity for $k\geq i$.
\end{proof}

\section{The weight filtration and a cohomological $P=W$ phenomenon}\label{section_weightfiltration}

In this section we use  the by now classical story of Deligne's weight filtration and its compatibility with the K\"unneth theorem to explicitly calculate the weight filtration for algebraic tori,
%Following \cite[Section 4]{FlorentinoSilva},
we  then explain how the weight filtration descends along finite group quotients. Theorem \ref{mainthm_P=W} follows by comparison with the analogous statements for the perverse filtration in Section \ref{section_perversefiltration}. We also prove Corollary \ref{maincor_curious}.

\subsection{Mixed Hodge structures on products}
\label{section_MHSproduct} Let $X$ be a quasi-projective variety. Deligne's classical results \cite{Deligne_HodgeII, Deligne_HodgeIII} (also see \cite{Durfee, PetersSteenbrink}) tell us that every cohomology group $H^i(X,\Q)$ admits a natural filtration 
\begin{equation*}
0=W_{-1}H^i(X,\Q)\subseteq W_{0}H^i(X,\Q) \subseteq \cdots \subseteq W_{2i}H^i(X,\Q)= H^i(X,\Q)
\end{equation*}
by linear subspaces such that the quotients
\begin{equation*}
\Gr_k^W H^i(X,\Q)=W_kH^i(X,\Q)/ W_{k-1}H^i(X,\Q)
\end{equation*}
carry a pure Hodge structure of weight $k$.
This association is functorial with respect to algebraic morphisms and compatible with cup product; hence it makes sense to write   
\begin{equation*}
W_kH^\ast(X,\Q)=\bigoplus_iW_kH^i(X,\Q) \ .
\end{equation*}
The weight filtration is also compatible with the K\"unneth formula, which means  
\begin{equation*}
W_k H^\ast(X\times Y)\cong \bigoplus_{s+t\leq k}W_sH^\ast(X) \otimes W_t H^\ast(X) 
\end{equation*}
for quasi-projective varieties $X$ and $Y$ (see \cite[Theorem 5.44]{PetersSteenbrink}). 

Recall now that for an algebraic torus $T=(\mathbb{C}^\ast)^m$  we have
\begin{equation*}
H^\ast(T,\Q)=\bigoplus_{i=0}^m\bigwedge^i\Q^m
\end{equation*}
by the K\"unneth formula,  generated in degree $1$. The following Proposition \ref{prop_weightproduct} describes the weight filtration on the cohomology of $T$.

\begin{proposition}\label{prop_weightproduct}
Let $T=(\mathbb{C}^\ast)^m$. Then the weight filtration is given by 
\begin{equation*}
W_{2k}H^\ast(T,\Q)=W_{2k+1} H^\ast(T,\Q)=\bigoplus_{i=0}^{k}\bigwedge^i\mathbb{Q}^m=\bigoplus_{i=0}^{k} H^i(T,\Q) \ .
\end{equation*}
\end{proposition}

Proposition \ref{prop_weightproduct} is well-known to experts in the field; we give a proof for lack of a suitable reference.

\begin{proof}
We first note that in the case $m=1$ we have $H^\ast(\C^\ast,\Q)=\Q\oplus \Q$, since $\C^\ast$ deformation retracts to $S^1$. Consider the inclusion $i\colon\C^\ast\hookrightarrow\P^1$. The pullback along $i$ induces a long exact sequence
\begin{equation*}
 \underbrace{H^1(\P^1,\Q)}_{=0}\longrightarrow \underbrace{H^1(\C^\ast,\Q)}_{=\Q}\longrightarrow \underbrace{H^2_{\{0,\infty\}}(\P^1,\Q)}_{=\Q^2}\longrightarrow \underbrace{H^2(\P^1,\Q)}_{=\Q}\longrightarrow \underbrace{H^2(\C^\ast,\Q)}_{=0} \ .
\end{equation*}
So the weight filtration on $H^\ast(\C^\ast,\Q)$ is given by
\begin{equation*}
W_0 H^\ast(\C^\ast,\Q) = \Q\ \qquad W_1 H^\ast(\C^\ast,\Q) = \Q \  \qquad W_2 H^\ast(\C^\ast,\Q) = \Q\oplus\Q 
\end{equation*} 
(see \cite[Example 3.4.9]{ElZeinLe_MHS} for details). So our result is true for $m=1$. For general $m\geq 1$ 
the K\"unneth formula is  compatible with the weight filtration and so we have
\begin{equation*}
W_{2k} H^i\big((\C^\ast)^m,\Q\big)= W_{2k+1} H^i\big((\C^\ast)^m,\Q\big)=\bigoplus_{i=0}^{k}\bigwedge^i\Q^m=\bigoplus_{i=0}^k H^i(T,\Q) \ .
\end{equation*}
\end{proof}

\subsection{Mixed Hodge structures and finite group quotients}\label{section_MHSquot}
Let $G$ be a finite group acting on a quasi-projective algebraic variety $X$. We write 
\begin{equation*}
H^\ast_G(X,\mathbb{Q})=H^\ast (EG\times_GX,\mathbb{Q})
\end{equation*}
for the equivariant cohomology ring with rational coefficients. Here $EG$ is the universal principal $G$-bundle over the classifying space $BG$ of $G$ and $EG\times_G X$ is the quotient of $EG\times X$ by the diagonal action of $G$. By the Vietoris--Begle theorem \cite[p. 344]{Spanier} the pullback along $EG\times_G X\rightarrow X/G$ induces a natural isomorphism
\begin{equation*}
H^\ast(X/G,\mathbb{Q})\xlongrightarrow{\sim} H^\ast_G(X,\Q) \ .
\end{equation*}
Moreover, using the Serre spectral sequence, one may deduce that $H^\ast_G(X,\Q)\cong H^\ast (X,\Q)^G$ and that the composition
\begin{equation*}
H^\ast(X/G,\mathbb{Q})\xlongrightarrow{\sim} H^\ast_G(X,\Q)\cong H^\ast (X,\Q)^G
\end{equation*}
is induced by the pullback along the (algebraic) quotient map $\pi\colon X\rightarrow X/G$ (see \cite[Proposition 4.3]{FlorentinoSilva} for details). Since $\pi$ is algebraic, the pullback homomorphism is compatible with the weight filtrations and thus we obtain
\begin{equation*}
W_kH^\ast(X/G,\Q)\cong \big(W_k H^\ast(X,\Q)\big)^G \ .
\end{equation*}

\subsection{Proof of Theorem \ref{mainthm_P=W}} Let $X=\C^g/L$ be a complex abelian variety of dimension $g$ and write $\widehat{X}$ for its dual. By Proposition \ref{prop_nonAbelianHodge}, both $M_{\textrm{Dol}}^r(X)$ and $M_{\textrm{Betti}}^r(X)$ are $r$-fold symmetric products. We have $M_{\textrm{Betti}}^r(X)=\Sym^r\big(M_{\textrm{Betti}}^1(X)\big)$ and Section \ref{section_MHSproduct} and \ref{section_MHSquot} show that the weight filtration is compatible with symmetric products; thus we find: 
\begin{equation*}
W_{k}H^\ast(M^r_{\textrm{Betti}}(X))\cong\big(W_{k} H^\ast\big((M^1_{\textrm{Betti}}(X))^r\big)\big)^{S_r} \ .
\end{equation*}
On the other hand, we have $M_{Dol}^r(X)=\Sym^r\big(M_{Dol}^1(X)\big)$ and by Proposition  \ref{thm_Pfiltration} the perverse filtration is compatible with symmetric powers:
\begin{equation*}
P_kH^*(M_{\textrm{Dol}}^r(X), \mathbb{Q})  \cong \big(P_kH^*\big((M^1_{\textrm{Dol}})^r, \mathbb{Q}\big) \big)^{S_r} \ .
\end{equation*}

Since the non-abelian Hodge correspondence is compatible with the symmetric products on both sides by Proposition \ref{prop_nonAbelianHodge}, it is thus enough to prove Theorem \ref{mainthm_P=W} in the case $r=1$. 
Here the non-abelian Hodge correspondence induces a real analytic isomorphism \begin{equation*}
(\C^\ast)^{2g}=M_{\textrm{Betti}}^1(X)\cong M_{\textrm{Dol}}^{1}(X)=\widehat{X}\times \mathbb{C}^g 
\end{equation*}
that induces an isomorphism
\begin{equation*}
H^\ast\big((\C^\ast)^{2g},\Q\big)=H^\ast\big(M_{\textrm{Betti}}^1(X)\big)\cong H^\ast\big(M_{\textrm{Dol}}^{1}(X)\big)=H^\ast\big(\widehat{X}\times \mathbb{C}^g \big)
\end{equation*}
of cohomology algebras. Write $T=(\mathbb{C}^\ast)^{2g}$. By Proposition \ref{prop_weightproduct}, the weight filtration on the left is given as
\begin{equation*}
W_{2k}H^\ast(T,\Q)=W_{2k+1} H^\ast(T,\Q)=\bigoplus_{i=0}^{k}\bigwedge^i\mathbb{Q}^{2g}=\bigoplus_{i=0}^{k} H^i(T,\Q) 
\end{equation*}
and, by Proposition \ref{prop_PfiltrationRankone}, the perverse filtration on the right as 
\begin{equation*}
P_{k} H^\ast(\widehat{X},\Q)=\bigoplus_{i=0}^{k} H^i(\widehat{X},\Q)=\bigoplus_{i=0}^{k}\bigwedge^i\mathbb{Q}^{2g} \ .
\end{equation*}
Thus, since the different subspaces in both the perverse and the weight filtration may be distinguished by degree,  the non-abelian Hodge correspondence induces an identification 
\begin{equation*}
P_{k} H^\ast\big(M_{\textrm{Dol}}^1(X),\Q\big)=W_{2k}H^\ast\big(M_{\textrm{Betti}}^1(X),\Q\big) =W_{2k+1}H^\ast\big(M_{\textrm{Betti}}^1(X),\Q\big)
\end{equation*}
for all $k\in\Z$.

\subsection{Curious Poincar\'e and curious hard Lefschetz} Let $X$ be a complex quasi-projective algebraic variety. Recall that  the \emph{mixed Hodge numbers} of $X$ are defined by
\begin{equation*}
h^{p,q;j}(X)=h^{p,q}\big( \Gr_{p+q}^WH^j(X,\Q) \big)\ , 
\end{equation*}
where $h^{p,q}$ on the right denotes the dimension of the $(p,q)$-th entry of the pure Hodge structure on $\Gr_{p+q}^WH^j(X,\Q)$. The \emph{mixed Hodge polynomial} is defined to be
\begin{equation*}
H(X;x,y,t):=\sum h^{p,q;j}(X)x^py^qt^j \ .
\end{equation*}
We say that $X$ is \emph{of Hodge--Tate type}, if $h^{p,q;j}(X)=0$, unless $p=q$. In this case the mixed Hodge polynomial $H(X;x,y,t)$ only depends on $xy$ and $t$ and we set 
\begin{equation*}
H(X;q,t):=H(X;\sqrt{q},\sqrt{q},t)
\end{equation*}
for a formal variable $q$.

\begin{proof}[Proof of Theorem \ref{mainthm_curious}]
Let $X=\C^g/L$ be a complex abelian variety of dimension $g\geq 1$. By \cite[Proposition 5.11]{FlorentinoSilva} the character variety $M_{\textrm{Betti}}^r(X)$ is of Hodge--Tate type and the mixed Hodge polynomial of $M_{\textrm{Betti}}^r(X)$ is given by 
\begin{equation*}
H\big(M_{\textrm{Betti}}^r(X); q,t)=\frac{1}{n!}\sum_{\sigma\in S_n}\big(\det(I_n-tqA_\sigma)\big)^{2g} \ ,
\end{equation*}
where $A_\sigma$ denotes the permutation matrix associated to $\sigma\in S_n$. Thus we can calculate
\begin{equation*}\begin{split}
H\big(M_{\textrm{Betti}}^r(X); \frac{1}{qt
^2},t)&=\frac{1}{n!}\sum_{\sigma\in S_n}\big(\det(I_n-\frac{1}{qt}A_\sigma)\big)^{2g}\\
&=\frac{1}{(qt)^{2gr}} \cdot \frac{1}{n!} \sum_{\sigma\in S_n}\big(\det(qtI_n-A_\sigma)\big)^{2g}\\
&=\frac{1}{(qt)^{2gr}} \cdot \frac{1}{n!} \sum_{\sigma\in S_n}\big(\det(I_n-qtA_\sigma)\big)^{2g}\\
&=\frac{1}{(qt)^{2gr}} \cdot H\big(M_{\textrm{Betti}}^r(X); q,t) \ .
\end{split}\end{equation*}
This proves Part (i). For Part (ii) we choose a hyperplane class $L\in H^2\big(M_{\textrm{Dol}}^r(X),\Q\big)$ and  note that the relative hard Lefschetz theorem (see \eqref{eq_relativehardLefschetz} in Section \ref{section_decomposition} for a version that applies to our situation) tells us that the $k$-fold application of $L$ defines an isomorphism 
\begin{equation*}
L^k\colon \Gr_{gr-k}^PH^\ast\big(M_{\textrm{Dol}}^r(X),\Q\big)\xlongrightarrow{\sim}\Gr_{gr+k}^PH^{\ast+2k}\big(M_{\textrm{Betti}}^r(X),\Q\big) \ .
\end{equation*}
Under the identification of the perverse with the weight filtration in Theorem \ref{mainthm_P=W} this becomes an isomorphism
\begin{equation*}
L^k\colon \Gr_{2gr-2k}^WH^\ast\big(M_{\textrm{Betti}}^r(X),\Q\big)\xlongrightarrow{\sim}\Gr_{2gr+2k}^WH^{\ast+2k}\big(M_{\textrm{Betti}}^r(X),\Q\big) \ ,
\end{equation*}
which is our claim.
\end{proof}

\subsection{An application: Higgs bundles of rank one on Jacobians} Let $X$ be a compact Riemann surface of genus $g\geq 1$. Theorem \ref{mainthm_P=W} implies a well-known version of the $P=W$ phenomenon for rank one Higgs bundles on $X$. Indeed, for $r=1$ and $d=0$, we canonically have
\begin{equation*}
M_{\textrm{Dol}}^{1,0}(X)\cong \Jac(X)\times H^0(X,K_X)\cong \Jac(X)\times H^{1,0}\big(\Jac(X)\big)\cong M_{\textrm{Dol}}^1\big(\Jac(X)\big)
\end{equation*}
by autoduality of $\Jac(X)$ and for the Betti moduli space
\begin{equation*}
M_{\textrm{Betti}}^{1,0}(X)\cong \Hom\big(\pi_1(X),\mathbb{C}^\ast\big)\cong\Hom(\pi_1(\Jac(X)),\mathbb{C}^\ast\big)=M_{\textrm{Betti}}^1(\Jac(X))\cong (\mathbb{C}^\ast)^{2g} \ ,
\end{equation*}
since canonically $\pi_1(X)^{ab}\cong\pi_1\big(\Jac(X)\big)$. Therefore we obtain:

\begin{corollary}\label{cor_rankone}
Let $X$ be a compact Riemann surface of genus $g\geq 1$. Then, under the identification \begin{equation*}
H^\ast\big(M_{\textrm{Dol}}^1(X),\Q\big)\cong H^\ast\big(M_{\textrm{Betti}}^1(X),\Q\big)\ ,
\end{equation*} we have 
\begin{equation*}
P_{k} H^\ast\big(M_{\textrm{Dol}}^{1,0}(X),\Q\big)=W_{2k}H^\ast\big(M_{\textrm{Betti}}^{1,0}(X),\Q\big)=W_{2k+1}H^\ast\big(M_{\textrm{Betti}}^{1,0}(X),\Q\big) 
\end{equation*}
for all $k\in\Z$.
\end{corollary}

The statement of Corollary \ref{cor_rankone} is of course well-known to experts; a somewhat different approach to its proof can be found in \cite[Theorem 5.9]{ZHA}. For a geometric $P=W$ phenomenon in rank one, we refer the reader to \cite[Theorem B]{MMS_geometricP=W}.

%%%%%%%%%%%%%%%%%%%%%%%%%%%%%%%%%%%%%%%%%%%%%%%%%%%%%%

%%%%%%%%%%%%%%%%%%%%%%%%%%%%%%%%%%%%%%%%%%%%%%%%%%%%%%

%%%%%%%%%%%%%%%%%%%%%%%%%%%%%%%%%%%%%%%%%%%%%%%%%%%%%%

%%%%%%%%%%%%%%%%%%%%%%%%%%%%%%%%%%%%%%%%%%%%%%%%%%%%%%

%%%%%%%%%%%%%%%%%%%%%%%%%%%%%%%%%%%%%%%%%%%%%%%%%%%%%%

%%%%%%%%%%%%%%%%%%%%%%%%%%%%%%%%%%%%%%%%%%%%%%%%%%%%%%

%%%%%%%%%%%%%%%%%%%%%%%%%%%%%%%%%%%%%%%%%%%%%%%%%%%%%%

%%%%%%%%%%%%%%%%%%%%%%%%%%%%%%%%%%%%%%%%%%%%%%%%%%%%%%

\bibliographystyle{amsalpha}
\bibliography{biblio}{}

\end{document}